\documentclass[10pt]{amsart}

\RequirePackage{newlfont}

\usepackage{amscd}
\usepackage{amsthm} % numbering of Lemmas and Theorems
\usepackage{amsmath}
\usepackage{amsfonts}
\usepackage{amssymb}
\usepackage{enumerate}
\usepackage{graphicx}
\usepackage{latexsym}
\usepackage{color}
\usepackage{bbm}

\usepackage{relsize} % To use \mathlarger for larger sizes

\newcommand{\bigint}{\mathlarger{\int}}

\newtheorem{theorem}{Theorem}[section]

\newtheorem{proposition}{Proposition}[section]
\newtheorem{lemma}{Lemma}[section]
\newtheorem{corollary}{Corollary}[section]

\newtheorem{definition}{Definition}[section]

\newtheorem{remark}{Remark}[section]
\newtheorem{statement}{Statement}[section]

\numberwithin{equation}{section}

\begin{document}

\title[Modular topologies on Vector Spaces]{Modular Topologies on Vector Spaces: Structure and Normability}

\author[M. A. Khamsi, J. Lang, O. M\'{e}ndez ]{Mohamed  A. Khamsi, Jan Lang,  Osvaldo M\'{e}ndez}

\address{Mohamed A. Khamsi\\ Department of Applied Mathematics and Sciences, Khalifa University, Abu Dhabi, UAE}
\email{mohamed.khamsi@ku.ac.ae}
\address{Jan Lang\\ Department of Mathematics, 100 Math Tower, 231 West 18th Ave., Columbus,
OH 43210-1174, USA; Faculty of Electrical Engineering, Department of Mathematics, Czech Technical University, Technicka 2, 166 27 Prague 6, Czech Republic}
\email{lang.162@osu.edu}

\address{Osvaldo M\'{e}ndez\\Department of Mathematical Sciences, The University of Texas at El Paso, El Paso, TX 79968, USA}
\email{osmendez@utep.edu}


\begin{abstract}
We investigate the topology generated by modular convergence on vector spaces endowed with a convex modular. Although modular convergence has long played a central role in the theory of modular function spaces, the topology it induces has, to the authors' best knowledge, not previously been studied as an independent mathematical object.

Our principal result establishes that the modular topology is compatible with the vector space structure if and only if the underlying modular satisfies the $\Delta_2$-condition. Consequently, the $\Delta_2$-condition admits a purely topological characterization. We also show that the Luxemburg norm topology is the weakest first-countable topology containing all scaled modular topologies. We apply the theory to variable exponent sequence and Lebesgue spaces beyond the classical $\Delta_2$-theory and present an application of our results to the minimization of the Dirichlet integral in $W^{1,p(x)}$ with unbounded exponent $p$, which allows us to solve an open boundary value problem.
\end{abstract}

\maketitle

\section{Introduction}

Modular spaces provide a flexible extension of classical Banach space theory.
They occur in Orlicz-type spaces, variable-exponent Lebesgue and Sobolev
spaces, and variational problems with non-standard
growth.  Two topological constructions must, however, be distinguished.  In
the classical construction one prescribes suitable modular balls as a
neighborhood system; see Musielak~\cite[Chapter~I, \S6]{M:1983},
Hajji~\cite{Hajji2013}, and the related topological-vector-space framework of
Kozlowski~\cite{Kozlowski2020,kozlowski_book}.  In the present paper we instead begin with
the sequential convergence
\[
x_j\xrightarrow{\rho}x
\quad\Longleftrightarrow\quad
\rho(x_j-x)\longrightarrow0
\]
and take its associated sequential topology in the sense of
Dudley~\cite{dudley}.  We call this topology the \emph{sequential modular
topology}.

Under the $\Delta_2$-condition the distinction disappears: modular
convergence and Luxemburg norm convergence agree, and the sequential modular
topology is the usual norm topology.  Without $\Delta_2$, modular balls need
not be open in the sequential topology.  Thus the results below complement,
rather than replace, the classical modular-ball topology.  Recent work on
modular topologies of Orlicz sequence spaces also uses balls as a local base;
see~\cite{Haryadi2025}.  Our emphasis is the topological modification forced
by the convergence relation itself when no ball-base axiom is imposed.

The principal abstract result, Theorem~\ref{Mainequivalence}, proves that the
sequential modular topology is compatible with the vector-space operations if
and only if the sequential $\Delta_2$-condition holds.  Equivalently, this is
precisely the case in which modular convergence agrees with Luxemburg norm
convergence.  We also study the scaled modulars
$\rho_\lambda(x)=\rho(\lambda x)$ and prove that the Luxemburg norm topology
is the weakest first-countable topology stronger than all associated
sequential modular topologies.

For variable-exponent sequence and Lebesgue spaces the failure of $\Delta_2$
produces a topology which is strictly weaker than the norm topology.  Under
the hypotheses stated in Section~\ref{lpn}, we prove that it remains Hausdorff
and separable, while modular balls have empty interior and the topology is not
first-countable.  We give explicit constructions in both the discrete and
continuous settings, and we exhibit countable dense subsets rather than
relying only on density of an uncountable smooth core.

The final section gives a variational application.  The argument is formulated
for the weighted gradient modular
$R_p(z)=\int_\Omega |z|^{p(x)}/p(x)\,dx$.  We include the completeness and
lower-exponent convergence statements needed for the direct method and state
the conclusion as uniqueness of the variational minimizer, together with its
Euler identity for the prescribed test space.

The paper is organized as follows.  Section~2 recalls modular spaces and the
Luxemburg norm.  Section~3 constructs the sequential modular topology.
Section~4 treats scaled topologies and their first-countable envelope.
Section~5 proves the $\Delta_2$ characterization.  Section~6 treats
variable-exponent spaces.  Section~7 develops the variational application to
the $p(\cdot)$-Dirichlet integral.

\thispagestyle{empty}

\section{Modular vector spaces}\label{modularvectorspaces}
Modular vector spaces provide a rich structure that allows for the exploration of aspects of functional analysis that are out of the reach of the theory of topological vector spaces.\\
In this section, the definition and basic properties of modular vector spaces are recalled. The reader is referred to \cite{KK} and \cite{M:1983} and the references therein for further details.

\begin{definition}\label{modular-vs}
A modular on a real vector space $X$ is a functional $\rho:X\rightarrow [0,\infty]$ such that
\begin{enumerate}\label{modularproperties}
\item[(1)] $\rho(u) = 0$ if and only if $u = 0$;
\item[(2)] $\rho(\alpha u) = \rho(u)$, if $|\alpha| = 1$;
\item[(3)] $\rho(\alpha u + (1-\alpha) v )\leq \alpha \rho(u) + (1-\alpha)\rho(v)$, for any $\alpha \in [0,1]$, and any $u, v \in X$.
\end{enumerate}
\end{definition}

In what follows, we assume that $\rho$ is left-continuous, meaning that for any $u \in X$, $\lim\limits_{\lambda \rightarrow \lambda_0^-}\ \rho(\lambda u) = \rho(\lambda_0 u)$, for each $\lambda_0>0.$\\

Let $X$ be a real vector space, and let $\rho:X \to [0,\infty]$ be a convex modular. Then,
$$X_{\rho}=\Big\{v\in X: \rho(\lambda v)<\infty \, \text{for some} \, \lambda>0\Big\}$$
is a vector subspace of $X$. It is straightforward to show that
$$X_{\rho}=\left\{v\in X: \lim\limits_{\lambda \to 0} \rho(\lambda v) = 0\right\}.$$ 

\noindent A fundamental role in the theory of modular vector spaces is played by the
$\Delta_2$-condition. This condition governs many of the basic structural
and topological properties of modular spaces and, as we shall see, its
failure leads to phenomena that are central to the present work.

\begin{definition}\label{Delta2}
A modular $\rho$ on a vector space $X$ is said to satisfy the $\Delta_2$-condition if for  every
sequence $(x_j)\subset X$ 
\[x_j \xrightarrow{\rho} 0\Rightarrow
2x_j \xrightarrow{\rho} 0.
\]
\end{definition}

\medskip

\noindent Two notable examples of modular vector spaces were introduced in the pioneering works of Orlicz \cite{orlicz1931} and Nakano \cite{nakano, nakano3}. Specifically:

\begin{definition}\label{deflp}
Fix a sequence $\mathbf{p}: 
= (p_n) \subset [1,\infty)$. On the vector space $X = \mathbb{R}^{\mathbb{N}}$ of all real sequences, define the functional $\rho_{\mathbf{p}}: X \to [0, \infty]$ by
$$\rho_{\mathbf{p}}((a_j)) := \textstyle \sum\limits_{j = 1}^{\infty} |a_j|^{p_j}.$$
Then $\rho_{\mathbf{p}}$ is a convex and left-continuous modular. Moreover, the associated modular vector space $X_{\rho_{\mathbf{p}}}$, denoted $\ell^{(p_n)}$, is 
$$\ell^{(p_n)} := \Big\{(a_j) \in  \mathbb{R}^{\mathbb{N}} : \textstyle \sum\limits_{j = 1}^{\infty} |\lambda a_j|^{p_j} < \infty \, \text{for some} \, \lambda > 0 \Big\}.$$

\end{definition}
Similarly,
\begin{definition}\label{Lp}
\noindent Let $\Omega\subseteq {\mathbb R}^n$ be an open set and $p:\Omega\rightarrow [1,\infty)$ be a Borel-measurable function. On the set of extended-real valued Borel-measurable functions on $\Omega$, ${\mathcal M}(\Omega)$, the functional
$$\rho_{p(\cdot)}(u):=\int\limits_{\Omega}|u(x)|^{p(x)}dx,$$ or shortly $\rho_p$,
is a left-continuous convex modular; the associated modular space is
$$L^{p(\cdot)}(\Omega):=\Bigg\{u\in {\mathcal M}(\Omega): \int\limits_{\Omega}|\lambda u(x)|^{p(x)}dx<\infty \,\,\text{for some}\; \lambda>0 \Bigg\}.$$
\end{definition}
For a measurable exponent we use the standard notation
\[
p^-:=\operatorname*{ess\,inf}_{x\in\Omega}p(x),
\qquad
p^+:=\operatorname*{ess\,sup}_{x\in\Omega}p(x).
\]
For an exponent sequence, $p^-:=\inf\limits_n p_n$ and
$p^+:=\sup\limits_n p_n$.
\noindent In this work the modular vector spaces \(\ell^{(p_n)}\) and \(L^{p(\cdot)}(\Omega)\) will be referred to as variable exponent spaces.

\begin{remark}{\normalfont  
Consider a convex modular $\rho$ on a vector space $X$. Let 
$$D:=\{x\in X_\rho:\ \rho(x)\leq 1\}.$$ 
Then it holds $X_{\rho}=\langle D\rangle$, meaning $X_\rho$ is the linear span of $D$. Indeed, if $x \in X_{\rho}$, either $\rho(\lambda x) \leq 1$ or $1 < \rho(\lambda x) < \infty$. In the latter case, and by the convexity
of $\rho$ and the fact that $\rho(0)=0$, it holds that
$$\rho\left(\frac{\lambda x}{\rho(\lambda x)}\right) \leq 1.$$
In both cases, $x$ belongs to $\langle D\rangle$. Conversely, if $x\in \langle D\rangle$, let $x = \textstyle \sum\limits_{j=1}^K\alpha_jx_j$ with $x_j \in D$ for $j = 1, \dots, K$. Since $D$ is balanced, considering $\alpha_j > 0$ for all $j$ suffices. Then it follows from the convexity of $\rho$ that
$$\rho\Big(\big(\textstyle \sum\limits_{j=1}^K \alpha_j\big)^{-1} x\Big) \leq 1,$$
proving the claim.}
\end{remark}

The concept of modular balls is of key importance in what follows.

\begin{definition}
For $a \in X$ and $\varepsilon > 0$, let  
$$B_{\rho,\varepsilon}(a) := \big\{y \in X : \rho(y - a) < \varepsilon\big\}.$$ 
In the sequel $B_{\rho,\varepsilon}(a)$ will be referred to as the modular ball of radius $\varepsilon$, centered at $a$.
\end{definition}

\begin{proposition}\label{X_rho-stability}
If $a \in X_{\rho}$, the convexity of $\rho$ implies that $B_{\rho,r}(a) \subseteq X_{\rho}$, for any $r > 0$.
\end{proposition}
\begin{proof}
By assumption, there exists $\lambda > 0$ such that $\rho(\lambda a) < \infty$. Without loss of generality, assume $0 < \lambda < 1$. Then, for any $y \in B_{\rho,r}(a)$,
\begin{align*}\rho\left(\frac{\lambda}{2} y\right) = \rho\left(\frac{\lambda}{2} (y - a) + \frac{\lambda}{2} a\right) &\leq \frac{1}{2} \Big( \rho(\lambda (y - a)) + \rho(\lambda a) \Big) \\
&< \frac{1}{2} \big(\lambda\ r + \rho(\lambda a)\big) < \infty,
\end{align*}
hence $y\in X_{\rho}$.
\end{proof}

\subsection{The Luxemburg norm}\leavevmode

Since each modular ball  $B_r(0):=\{x\in X_\rho:\rho(x)<r\}$ is a convex, balanced, and absorbent set in $X_\rho$, the Minkowski functional
\[
\mu_{B_r}(z):=\inf\left\{\lambda>0:\rho(\lambda^{-1}z)\le r\right\}
\]
defines a norm on $X_\rho$.  It can be easily shown that on account of the convexity of $\rho$, for any two modular balls $B_{r_1}(0)$ and $B_{r_2}(0)$ with $r_1 \leq r_2$, their corresponding Minkowski functionals satisfy
$$\mu_{B_{r_2}}(x) \leq \mu_{B_{r_1}}(x) \leq \frac{r_2}{r_1} \mu_{B_{r_2}}(x),$$
for all $x \in X_{\rho}$. Due to this equivalence of norms, it is customary to focus on the Minkowski functional of the unit modular ball, denoted as $\mu_{B_1}$. We also adopt the notation
$$\mu_{B_1}(x) = \|x\|_{\rho} := \inf \Big\{\lambda > 0 : \rho(\lambda^{-1}x) \leq 1 \Big\}.$$
The norm $\|\cdot\|_{\rho}$ is referred to as the Luxemburg norm on the modular space $X_{\rho}$ \cite{Luxemburg}.
Proposition \ref{standard} is well known.
\begin{proposition}\label{standard}\cite{KR}
Consider the modular vector space $X_{\rho}$.
\begin{enumerate}
\item [$(i)$] Since $\rho$ is left-continuous, for any $x \neq 0$, then $\rho\left(\|x\|^{-1}_\rho x \right) \leq 1$.
\item[$(ii)$] $\rho(x) \leq 1 \iff \|x\|_{\rho} \leq 1$.
\item [$(iii)$] If $\|x\|_{\rho} \leq 1$, then $\rho(x) \leq \|x\|_{\rho}$.
\item[$(iv)$] If, in addition, $\rho$ is right-continuous, then $\rho(x) < 1 \iff \|x\|_{\rho} < 1$.
\end{enumerate}
\end{proposition}
\begin{proof}
The proof follows directly from the definitions and properties of the modular $\rho$.
\end{proof}

\begin{remark}{\normalfont  
In general, the condition $\rho(x_0) < 1$ does not imply $\|x_0\| < 1$. For instance, take the domain $\Omega = \left(0, \frac{1}{2}\right)$ and the function $p(x) = x^{-1} $. Consider the space $L^{p(\cdot)}((0,\frac{1}{2}))$ as in Definition \ref{Lp}, that is, the modular is given by  $\rho(u) = \bigint_0^{1/2} |u(x)|^{p(x)} dx$ .
Let $\theta > 1$. Observe that $\lim_{x\rightarrow 0^+} \big(\theta^{x^{-1}} - x^{-1}\big) = \infty$.  Thus, there exists a point $x_0$ in the interval $0 < x_0 < \frac{1}{2}$ such that for $0 < x < x_0$, we have $\theta^{x^{-1}} - x^{-1} > \frac{1}{2}.$
It follows that
$$\int_0^{\frac{1}{2}} \theta^{x^{-1}} dx \geq \int_0^{x_0} \left( \frac{1}{2} + \frac{1}{x} \right) dx = \infty.$$
Consequently, we have $\rho(1) = \frac{1}{2} < 1$. However, for any $\lambda$ such that $0 < \lambda < 1$, $\rho\left( \frac{1}{\lambda} \right) = \infty$, which implies that $\|1\|_{\rho} = 1$. This example illustrates that modular balls might not be norm-open.\\
Moreover, it shows that the inequality in
Proposition~\ref{standard}(i) can be strict when the modular is not
right-continuous.}
\end{remark}

\section{Modular topologies}\label{modulartopologies}
It is well known that under appropriate conditions, an abstract notion of convergence on a non-empty set $X$ generates a topology on $X$ that is naturally compatible with the given convergence. This phenomenon has been thoroughly investigated in  \cite{dudley}. In this section we present the basic properties of the topology generated by modular convergence in a vector space, which will be referred to as the \emph{modular topology}. The results we derive are specific to modular convergence and are not straightforward consequences of \cite{dudley}. In particular, we study the relation of the modular topology with the algebraic structure of the space (Proposition \ref{basic-properties} and Proposition \ref{subspace}). These properties have profound implications in the study of boundary value problems, as shall be outlined in Section \ref{applications}. A deeper look at the modular topology will follow in the subsequent sections.
The following definition is standard.
\begin{definition}\label{Defmodulcon}
A sequence $(x_n)\subset X$ is said to $\rho$-converge to
$x\in X$, written $x_n\xrightarrow{\rho}x$, if
\[
\rho(x_n-x)\longrightarrow0 \,\,\text{as}\,\, n\rightarrow \infty.
\] 
\end{definition}
The following lemma is straightforward, yet fundamental in the sequel.
\begin{lemma}\label{lconvergence}
Let $\rho$ be a modular on a vector space $X$. Then
\begin{enumerate}
\item[$(i)$] Every constant sequence $x_n=x$ $\rho$-converges to $x$.
\item[$(ii)$] A sequence has at most one $\rho$-limit.
\item[$(iii)$] Every subsequence of a $\rho$-convergent sequence is convergent to the same $\rho$-limit.

\item[$(iv)$] If $(x_n)$ does not $\rho$-converge to $x$, then it has a subsequence no further subsequence of which $\rho$-converges to $x$.
\end{enumerate}
\end{lemma}
\begin{proof}
Only $(ii)$ requires comment.  If $x_n\xrightarrow{\rho}x$ and
$x_n\xrightarrow{\rho}y$, then
\[
\rho\left(\frac{x-y}{2}\right)
\le \frac12\rho(x-x_n)+\frac12\rho(x_n-y)\longrightarrow0.
\]
Thus $\rho((x-y)/2)=0$, and hence $x=y$.  For (iv), if convergence to
$x$ fails, choose $\varepsilon>0$ and a subsequence $(x_{n_k})$ with
$\rho(x_{n_k}-x)\ge\varepsilon$ for every $k$.

\end{proof}
\begin{corollary}\label{l*convergence}

Lemma~\ref{lconvergence} shows that modular convergence is an
$L^{\ast}$-convergence in the sense of \cite{dudley}.  The associated
sequential topology is described explicitly below.

\end{corollary}

\begin{remark}{\normalfont 
Our approach here differs from the approach in \cite{dudley}. We undertake the study of the topology generated through modular convergence and characterize the open sets in terms of modular balls, which, surprisingly, need not be open themselves. This description is novel and necessary for the applications to the study of nonlinear boundary value problems, as shown in \cite{AOJA}.
}
\end{remark}
The following result is straightforward:

\begin{proposition}\label{rho-topology}

For $A\subset X$, the following are equivalent:
\begin{enumerate}
\item Every $\rho$-convergent sequence in $A$ has its limit in $A$;
\item For every $x\notin A$, there exists $\varepsilon>0$ such that
$B_{\rho,\varepsilon}(x)\cap A=\varnothing$.
\end{enumerate}

\end{proposition}
\begin{proof}

If (2) fails at some $x\notin A$, choose
$x_n\in A\cap B_{\rho,1/n}(x)$.  Then $x_n\xrightarrow{\rho}x$, contrary
to (1).  Conversely, if (2) holds and $x_n\in A$ $\rho$-converges to
$x\notin A$, then $x_n$ eventually belongs to a modular ball about $x$
that is disjoint from $A$, a contradiction.

\end{proof}

The modular $\rho$ defines a topology $\tau_\rho$ on $X$ as follows.

\begin{definition}\label{tau-open}
A set $A\subset X$ is $\tau_\rho$-open if for every $x\in A$
there exists $\varepsilon>0$ such that $B_{\rho,\varepsilon}(x)\subset A$. It is a routine exercise to show that the collection $\tau_{\rho}$ is a topology on $X$.
\end{definition}

\begin{remark}
{\normalfont
We emphasize that the modular balls $B_{\rho,\varepsilon}(x)$ appearing
in Definition \ref{tau-open} are not asserted to be $\tau_\rho$-open.
They characterize the open sets of the modular topology, but they need
not themselves be open. This distinction will play an important role
in what follows.}
\end{remark}

From Proposition \ref{rho-topology}, it is clear that { $C \subset X$} is $\tau_\rho$-closed if and only if, for any sequence $(x_n)\subseteq C$ which $\rho$-converges to $x$, it holds $x \in C$.

\begin{corollary}
It follows from Proposition \ref{X_rho-stability} that the unique $\tau_{\rho}$-limit $x$ of a $\tau_{\rho}$-convergent sequence $(x_{n})\subset X_{\rho}$, must be in $X_{\rho}$. Therefore, $X_{\rho}$ is a $\tau_{\rho}$-closed subspace of $X$.
\end{corollary} 

\begin{remark}{\normalfont
It is obvious that for any { $x \in X_\rho$}, the complement ${ X_\rho}\setminus \{x\}$ is $\tau_\rho$-open. Hence, $\tau_\rho$ is a $T_1$ topology.
}
\end{remark}

The next theorem follows from Corollary \ref{l*convergence} and  \cite[Theorem 2.1]{dudley}. For the sake of completeness, we include a simple self-contained proof in the modular setting.

\begin{theorem}\label{convergencecharacterization}
For any sequence $(x_n)$ in { $X$}, it holds that $(x_n) \overset{\rho}{\rightarrow} x$ if and only if $(x_n) \rightarrow x$ in $\tau_\rho$ (or shortly  $(x_n) \overset{\tau_\rho}{\rightarrow} x$).
\end{theorem}
\begin{proof}
Assume first that $(x_n) \overset{\rho}{\rightarrow} x$. Let $\mathcal{O}$ be a $\tau_\rho$-open set that contains $x$. There exists $\varepsilon > 0$ such that $B_{\rho, \varepsilon}(x) \subset \mathcal{O}$. Since $(x_n) \overset{\rho}{\rightarrow} x$, there exists $n_0 \geq 1$ such that $x_n \in B_{\rho, \varepsilon}(x)$ for all $n \geq n_0$. Hence, $x_n \in \mathcal{O}$ for all $n \geq n_0$, that is, $(x_n) \overset{\tau_\rho}{\rightarrow} x$. Now assume that $(x_n) \overset{\tau_\rho}{\rightarrow} x$. Suppose that $(x_n)$ does not $\rho$-converge to $x$. Then, there exist $\varepsilon_0 > 0$ and a subsequence $(x_{n_k})$ such that $\rho(x - x_{n_k}) \geq \varepsilon_0$ for all $k \geq 1$. Then, either the subsequence $(x_{n_k})$ has a $\rho$-convergent subsequence or no subsequence of $(x_{n_k})$ $\rho$-converges. In the first case, let $(y_j)$ be a subsequence of $(x_{n_k})$ such that $y_j \overset{\rho}{\rightarrow} y \in X_\rho$. Clearly, $x \neq y$. Then, a straightforward argument shows that, by definition, the set
$$C = \{y_j : j \in \mathbb{N}\} \cup \{y\}$$
is $\tau_\rho$-closed and $x \notin C$. Since $(y_j) \overset{\tau_\rho}{\rightarrow} x$, we have $x \in C$, which is a contradiction. On the other hand, if no subsequence of $(x_{n_k})$ $\rho$-converges to any point, the set
$$A = \{x_{n_k} : k \in \mathbb{N}\}$$
is $\tau_\rho$-closed in $X_\rho$ and does not contain $x$. Its complement, $X_\rho \setminus A$, is thus $\tau_\rho$-open and contains $x$. Hence, $X_\rho \setminus A$ must contain $x_{n_k}$ for sufficiently large $n_k$, leading to a contradiction.
\end{proof}

Theorem \ref{convergencecharacterization}, along with the uniqueness of the \(\rho\)-limit, gives the following result:

\begin{corollary}
Any sequence in { $X$} can $\tau_{\rho}$-converge to at most one limit.
\end{corollary}

Note that $\tau_{\rho}$ is a topology on $X$. By  $\overline{\tau_{\rho}}$ we denote the subspace topology induced by $\tau_{\rho}$ on $X_{\rho}$. 
\\On account of Proposition \ref{X_rho-stability} it can be quickly seen that  $A\subseteq X_{\rho}$ is $\overline{\tau_{\rho}}$-open iff for any $a\in A$ there exists $\delta>0$ such that  $\{y\in X_{\rho}:\rho(a-y)<\delta\}\subset A$.\\
Thus:
\begin{corollary}\label{technicalcorollary}
It holds: $\overline{\tau_{\rho}}=2^{X_{\rho}}\cap \tau_{\rho}$.
\end{corollary}
From now on, all statements refer to the topological space $(X_{\rho},\overline{\tau_{\rho}})$. By virtue of the preceding remark, any subset of $X_{\rho}$ is $\overline{\tau_{\rho}}$-open if and only if it is $\tau_{\rho}$-open. In the interest of notational simplicity, we will slightly abuse the notation and use  $\tau_{\rho}$ instead of $\overline{\tau_{\rho}}$ to denote the subspace topology on $X_{\rho}$.\\

\begin{definition}
Let $\overline{A}^{\rho}$ denote the closure of $A\subseteq X_\rho$
in the modular topology $\tau_\rho$, i.e.,
\[
\overline{A}^{\rho}
:=
\bigcap_{\substack{A\subseteq W\\ W\ \tau_\rho\text{-closed}}} W.
\]
\end{definition}

\smallskip

\begin{definition}\label{fatouproperty}
A modular $\rho$ on a vector space $X$ is said to satisfy the
\emph{Fatou property} if whenever $(y_j)\xrightarrow{\rho}y\in X_\rho$,
it holds that
\[
\rho(y)\leq\liminf_{j\to\infty}\rho(y_j).
\]
\end{definition}

\noindent It follows directly from the definition of the Fatou property and the characterization of closed sets in the modular topology that $\rho$
satisfies the Fatou property if and only if, for every $x\in X_\rho$
and $\varepsilon>0$, the modular ball
\[
\overline{B}_{\rho,\varepsilon}(x)
=
\{v\in X_\rho:\rho(x-v)\leq\varepsilon\}
\]
is closed in the modular topology $\tau_\rho$.\\

\smallskip

The next two propositions reveal the connection between the modular topology $\tau_{\rho}$ and the vector space structure.

\begin{proposition}\label{basic-properties}
The following properties hold:
\begin{enumerate}
\item If $\mathcal{O}$ is a $\tau_\rho$-open subset of $X_\rho$, then $\mathcal{O} + x = \{u + x; u \in \mathcal{O}\}$ is also $\tau_\rho$-open, for any $x \in X_\rho$. Hence, $\mathcal{O}_1 + \mathcal{O}_2$ is $\tau_\rho$-open provided either $\mathcal{O}_1$ or $\mathcal{O}_2$ is $\tau_\rho$-open.
\item If $\mathcal{O}$ is $\tau_\rho$-open and $\alpha \geq 1$, then $ \alpha  \mathcal{O}$ is also $\tau_\rho$-open.
\item For any $x \in \overline{A}^\rho$ and any $\tau_\rho$-open subset $\mathcal{O}$ such that $x \in \mathcal{O}$, we have $\mathcal{O} \cap A \neq \emptyset$.
\item If $A$ is convex, then $\overline{A}^\rho$ is convex.
\end{enumerate}
\end{proposition}
\begin{proof}\leavevmode
\begin{enumerate}
\item Let $y \in \mathcal{O} + x$, then $y - x \in \mathcal{O}$. Since $\mathcal{O}$ is $\tau_\rho$-open, there exists $\varepsilon > 0$ such that $B_{\rho,\varepsilon}(y - x) \subset \mathcal{O}$. Clearly, $B_{\rho,\varepsilon}(y) \subset \mathcal{O} + x$. As for $\mathcal{O}_1 + \mathcal{O}_2$, note that 
\[\mathcal{O}_1 + \mathcal{O}_2 = \bigcup_{y \in \mathcal{O}_2} \mathcal{O}_1 + y = \bigcup_{x \in \mathcal{O}_1} \mathcal{O}_2 + x,\]
yielding the desired conclusion.
\item Let $x \in  \alpha  \mathcal{O}$. Then $ \beta  x \in \mathcal{O}$, where $\beta = 1/\alpha \in (0,1]$. Since $\mathcal{O}$ is $\tau_\rho$-open, there exists $\varepsilon > 0$ such that $B_{\rho,\varepsilon}( \beta  x) \subset \mathcal{O}$. For any $y \in B_{\rho,\alpha \varepsilon}(x)$, we have
\[\rho\left( \beta  x -  \beta  y\right) \leq  \beta  \rho(x - y) <  \beta   \alpha  \varepsilon = \varepsilon,\]
i.e., $ \beta  y  \in B_{\rho,\varepsilon}( \beta  x) \subset \mathcal{O}$. Hence $y \in  \alpha  \mathcal{O}$, which forces $B_{\rho, \alpha \varepsilon}(x) \subset \alpha \ \mathcal{O}$. Therefore, $\alpha \ \mathcal{O}$ is $\tau_\rho$-open. 
\item Suppose $\mathcal{O} \cap A = \emptyset$. Then $A \subset \mathcal{O}^c = X_\rho \setminus \mathcal{O}$. Since $\mathcal{O}$ is $\tau_\rho$-open, it follows that $\mathcal{O}^c$ is $\tau_\rho$-closed. Hence, $\overline{A}^\rho \subset \mathcal{O}^c$, i.e., $\mathcal{O} \cap \overline{A}^\rho = \emptyset$. This contradicts the assumption that $x \in \mathcal{O} \cap \overline{A}^\rho$.
\item Suppose $A$ is convex. Let $x, y \in \overline{A}^\rho$ and $\alpha \in (0, 1)$.  Assume $\alpha x + (1 - \alpha) y$ is not in $\overline{A}^\rho$.  Since $\mathcal{O} = X_\rho \setminus \overline{A}^\rho$ is $\tau_\rho$-open and $\alpha x + (1 - \alpha)y \in \mathcal{O}$, we can deduce that
\[x \in \mathcal{O}_1 =  \beta  \mathcal{O} - (\beta -1)\ y,\]
where $\beta = 1/\alpha$.  Since $\mathcal{O}_1$ is $\tau_\rho$-open, by the previously proven properties, $A \cap \mathcal{O}_1 \neq \emptyset$. Let $a \in A \cap \mathcal{O}_1$. Then there exists $x_0 \in \mathcal{O}$ such that
\[a =  \beta  x_0 - (\beta -1)\ y,\]
implying
\[y = \frac{1}{1 - \alpha} x_0 - \frac{\alpha}{1 - \alpha} a \in \mathcal{O}_2 = \frac{1}{1 - \alpha} \mathcal{O} - \frac{\alpha}{1 - \alpha} a.\]
Thus, $A \cap \mathcal{O}_2 \neq \emptyset$. Let $b \in A \cap \mathcal{O}_2$. Since $b \in \mathcal{O}_2$, there exists $y_0 \in \mathcal{O}$ such that
\[b = \frac{1}{1 - \alpha} y_0 - \frac{\alpha}{1 - \alpha} a,\]
implying $y_0 = \alpha a + (1 - \alpha)b \in A$, by the convexity of $A$. Therefore, $y_0 \in A \cap \mathcal{O}$ which contradicts the assumption that $\mathcal{O} \cap A = \emptyset$.
\end{enumerate}
\end{proof}

The preceding result has the following fundamental consequence:

\begin{proposition}\label{subspace}
Let $A$ be a vector subspace of $X_{\rho}$. Then $\overline{A}^\rho$ is a $\tau_\rho$-closed vector subspace of $X_{\rho}$.
\end{proposition}
\begin{proof}
It will be shown first that if $x, y \in \overline{A}^\rho$, then $x + y \in \overline{A}^\rho$. Suppose not, i.e., $x + y \in \mathcal{O} = X_\rho \setminus \overline{A}^\rho$. Then $x \in (\mathcal{O} - y)$, which is $\tau_\rho$-open. This forces $A \cap (\mathcal{O} - y) \neq \emptyset$. Let $a \in A \cap (\mathcal{O} - y)$. There exists $x_0 \in \mathcal{O}$ such that $a = x_0 - y$, implying $y = x_0 - a \in (\mathcal{O} - a)$. Again, since $\mathcal{O} - a$ is open, it follows that $A \cap (\mathcal{O} - a) \neq \emptyset$. Let $b \in A \cap (\mathcal{O} - a)$. Then there exists $y_0 \in \mathcal{O}$ such that $b = y_0 - a$, implying $y_0 = b + a$. Since $A$ is a subspace, we have $b + a \in A$. Therefore, the assumption implies that $y_0 \in A \cap \mathcal{O}$, contradicting the fact that $A \cap \mathcal{O} = \emptyset$. \\
Next, observe that if $x \in \overline{A}^\rho$ and $\alpha \in \mathbb{R}$, then $ \alpha  x \in \overline{A}^\rho$. Without loss of generality, assume $\alpha \neq 0$.  Take first $\alpha > 0$. The first part of the proof shows that $k\ x \in \overline{A}^\rho$ for any $k \in \mathbb{N}$. Thus, it may be assumed that $\alpha$ is not an integer. In this case, there exists $k \in \mathbb{N}$ such that $k < \alpha < k + 1$. This implies the existence of $\theta \in (0,1)$ such that $\alpha = \theta k + (1 - \theta)(k + 1)$. Hence,
\[ \alpha  x = \theta\ k\ x + (1 - \theta)(k + 1)\ x.\]
Hence, $\alpha \ x \in \overline{A}^\rho$ since $\overline{A}^\rho$ is convex by Proposition \ref{basic-properties}.\\
Finally, we show that if $x \in \overline{A}^\rho$, then $-x \in \overline{A}^\rho$. This follows from similar reasoning and the fact that if $\mathcal{O}$ is $\tau_\rho$-open, then $-\mathcal{O}$ is also $\tau_\rho$-open. To see this, let $\mathcal{O}$ be a $\tau_\rho$-open subset of $X_\rho$. Let $y \in -\mathcal{O}$. Then $-y \in \mathcal{O}$, implying the existence of $\varepsilon > 0$ such that $B_{\rho,\varepsilon}(-y) \subset \mathcal{O}$. Using the properties of the modular, it is easily seen that $z \in B_{\rho,\varepsilon}(-y)$ if and only if $-z \in B_{\rho,\varepsilon}(y)$. Therefore, $B_{\rho,\varepsilon}(y) \subset -\mathcal{O}$, completing the proof that $-\mathcal{O}$ is $\tau_\rho$-open.
\end{proof}

\section{The Luxemburg norm topology as the first-countable envelope of scaled modular topologies}\label{lambda}

Modular convergence is defined in \cite[Definition 5.1]{M:1983} in the following way: 
A sequence $(x_j)$ in a modular space $(X,\rho)$ converges to $x$ iff

\begin{statement}\label{Mdefinition} there exists $\lambda>0$ such that $\rho(\lambda(x_j-x))\rightarrow 0$ as $j\rightarrow \infty.$
\end{statement}
Observe that for any $\lambda>0$, the functional $\rho_{\lambda}(z):=\rho(\lambda z)$ is a modular on $X$ and thus, in terms of our terminology, Statement \ref{Mdefinition} is equivalent to 
\begin{statement}\label{Odefinition}
for some $\lambda>0$, $\rho_{\lambda}(x_j-x)\rightarrow 0$ as $j\rightarrow \infty$.
\end{statement}
The latter can be rephrased as requiring that for some $\lambda>0$, $x_j\overset{\rho_{\lambda}}\rightarrow x$ as $j\rightarrow \infty$.   \\
Thus, the modular topology introduced in the previous section corresponds to $\lambda=1$ in the natural family $(\tau_{\lambda})$ of $\rho_{\lambda}$-topologies. This is not an artificial particularization. The consideration of this special case will be justified by the fundamental role it plays in the treatment of boundary value problems, illustrated in Section \ref{applications}.\\

A fundamental question is how the scaled sequential topologies relate to the
Luxemburg norm topology.  Theorem~\ref{norminitial} identifies the latter as
the weakest first-countable topology stronger than every $\tau_\lambda$.
Although the Luxemburg norms generated by the $\rho_\lambda$ are proportional,
the sequential topologies are strictly ordered when $\Delta_2$ fails, as proved
below. This phenomenon underlines the role of the modular topology as a powerful analytical resource in situations that are out of reach of the  tools associated to the normed space structure.\\
It is straightforward to show that, as anticipated in the paragraph below Statement \ref{Odefinition} above, for each $\lambda > 0$, \(\rho_{\lambda}: X \to [0, \infty]\) given by \(\rho_{\lambda}(x) = \rho(\lambda x)\) is a convex modular on \(X\). Moreover, for any \(\lambda > 0\), we have
\[
X_{\rho_{\lambda}} := \{x \in X: \rho_{\lambda}(\alpha x) < \infty \, \text{for some}\, \alpha > 0\} = X_{\rho},
\]
i.e., all the modulars \(\rho_{\lambda}\) define the same modular vector space \(X_{\rho}\). 

Let $\tau_\lambda$ be the sequential topology associated with $\rho_\lambda$.
If $0<\lambda_1<\lambda_2$, convexity gives
\[
\rho(\lambda_1 z)
\le \frac{\lambda_1}{\lambda_2}\rho(\lambda_2 z),
\]
so $\tau_{\lambda_1}\subseteq\tau_{\lambda_2}$.
\begin{proposition}\label{scaled-strictness}
If $\rho$ fails $\Delta_2$, then
$\tau_{\lambda_1}\subsetneq\tau_{\lambda_2}$ whenever
$0<\lambda_1<\lambda_2$.
\end{proposition}
\begin{proof}
Suppose $\tau_{\lambda_1}=\tau_{\lambda_2}$ and put
$a=\lambda_2/\lambda_1>1$.  Apply the equality of the two topologies to
$y_j=x_j/\lambda_1$.  By Theorem~\ref{convergencecharacterization},
for every sequence $(x_j)$,
\[
\rho(x_j)\to0\quad\Longleftrightarrow\quad \rho(ax_j)\to0.
\]
Applying this equivalence successively yields
$\rho(a^m x_j)\to0$ for every $m$.  Choose $m$ with $a^m\ge2$.
Then convexity gives
\[
\rho(2x_j)\le \frac{2}{a^m}\rho(a^m x_j)\longrightarrow0,
\]
which is the $\Delta_2$-condition.  The contrapositive proves strictness.
\end{proof}

The family $(\tau_{\lambda})$ provides a collection of topologies that capture modular convergence at different scales. However, without the $\Delta_2$-condition, none of these topologies individually yields a structure fully compatible with the linear operations on $X$. This naturally leads to the question of whether one can combine this family into a single topology that better reflects the algebraic structure of the space.\\
To address this question, we consider two extremal topologies on $X$ associated with the family $(\tau_\lambda)$: the initial topology $\tau^i$ and the final topology $\tau^f$.

\begin{definition}
Let $\rho$ be a convex modular on a vector space $X$.
\begin{enumerate}
\item The final topology associated with the family $(\tau_\lambda)$ is the strongest among the topologies $\tau$ that make every inclusion $j_{\lambda}:(X,\tau_{\lambda})\hookrightarrow (X,\tau)$ continuous. It is
given by
\[
\tau^f:=\bigcap_{\lambda>0}\tau_\lambda.
\]

\item The initial topology, denoted by $\tau^i$, is the weakest topology
on $X$ that is stronger than every $\tau_\lambda$, $\lambda>0$. Equivalently, $\tau^i$ is the weakest among the topologies $\tau$ that make every inclusion $i_{\lambda}:(X,\tau)\rightarrow(X,\tau_{\lambda})$ continuous.
\end{enumerate}
\end{definition}

\noindent These two extremal constructions encode complementary aspects of the
family $(\tau_\lambda)$. The topology $\tau^i$ reflects the common
structure generated by all scales, while $\tau^f$ captures the topology
common to all members of the family. As we shall see, these topologies
play a crucial role in recovering partial compatibility between modular
convergence and the linear structure of $X$, and they provide a natural
bridge to the Luxemburg norm topology.\\

The open subsets for both topologies \(\tau^f\) and \(\tau^i\) are characterized in the following proposition:

\begin{proposition}\label{tau^i-tau^f-characterization}
The following hold:
\begin{enumerate}
\item The open sets in \(\tau^f\) are those subsets \(Y \subseteq X_{\rho}\) that are \(\tau_{\lambda}\)-open for every \(\lambda > 0\).
\item The open sets in \(\tau^i\) consist of arbitrary unions of finite intersections of the form \(A_{\lambda_1} \cap A_{\lambda_2}\cap \dots \cap A_{\lambda_N}\), where \(A_{\lambda_j} \in \tau_{\lambda_j}\). By virtue of the inclusion \(\tau_{\alpha} \subseteq \tau_{\beta} \) whenever \(\alpha \leq  \beta \), given any open set $A$ in \(\tau^i\) there exist $J\subseteq (0,\infty)$ such that
\[
A=\bigcup_{\lambda \in J } A_{\lambda},
\]
where each \(A_{\lambda}\) is open in \(\tau_{\lambda}\).
\end{enumerate}
\end{proposition}
\begin{proof}
The first assertion is the definition of the intersection topology.
For $(2)$, observe that any open set in $\tau^i$ is a union of finite intersections of open sets in some $\tau_{\lambda}$, so monotonicity gives the final
single-scale union representation.
\end{proof}
\begin{remark}\label{open-charac-tau^i}
{\normalfont Let $\mathcal O\in\tau^i$.  By Proposition
\ref{tau^i-tau^f-characterization}, write
$\mathcal O=\bigcup_{\lambda\in J}A_\lambda$ with
$A_\lambda\in\tau_\lambda$.  For $n\in\mathbb N$ put
\[
\mathcal O_n=\bigcup_{\lambda\in J\cap(0,n]}A_\lambda.
\]
Then $\mathcal O_n\in\tau_n$, the family is increasing, and
$\mathcal O=\bigcup_{n=1}^\infty\mathcal O_n$.}
\end{remark}

It is a routine exercise to verify that a set \(C \subseteq X_{\rho}\) is \(\tau^f\)-closed if and only if \(C\) is \(\tau_{\lambda}\)-closed for every \(\lambda > 0\), i.e., given any sequence \((x_j) \subseteq C\) such that, for some \(\lambda > 0\), \(\rho(\lambda (x_i - x)) \to 0\) as \(j \to \infty\), we have \(x \in C\). Since \(\tau^f\)-open subsets are also \(\tau_{\lambda}\)-open for every \(\lambda > 0\), we obtain the following fact:

\begin{proposition}
If $(x_j)\subset X_{\rho}$ and there exists $\lambda>0$ such that $(x_j)\overset{\tau_\lambda}{\rightarrow}x$, then $(x_j)\overset{\tau^f}{\rightarrow}x$.
\end{proposition}

Some algebraic properties of $\tau^i$ will be discussed next:

\begin{proposition}\label{tau^i-algebra}
The topology $\tau^{i}$ on $X_\rho$ is stable under addition and scalar multiplication. Specifically, if $A\in \tau^{i}$, $B\in \tau^{i}$, and $r\in {\mathbb R}$, with $r \neq 0$, then $A+B\in \tau^{i}$ and $rA\in \tau^{i}$ (however, $\tau_i$ is not necessarily a $TVS$ topology.)
\end{proposition}
\begin{proof}
Observe first that if $A\in \tau^{i}$ and $B\in \tau^{i}$, then $A+B\in \tau^{i}$. For, according to Remark \ref{open-charac-tau^i}, $A=\bigcup\limits_{n \geq 1}A_n$ and $B=\bigcup\limits_{n \geq 1}B_n$, where $A_n$ and $B_n$ are in $\tau_n$ for all $n \geq 1$. Using Proposition \ref{basic-properties}, we know that $A_n + B_n \in \tau_n$ for all $n \geq 1$. It follows that $A + B = \bigcup\limits_{n \geq 1}(A_n + B_n)$, which proves the desired result.

On the other hand $rA$ is $\tau^{i}$-open for any $r \in \mathbb{R}$, $r \neq 0$, provided that $A$ is $\tau^{i}$-open. To see this, assume $r > 0$. Since
\[
\rho_{\frac{\lambda}{r}}(x-w)= \rho\left(\frac{\lambda}{r} (x - w)\right) = \rho_\lambda \left(\frac{x}{r} - \frac{w}{r}\right),
\]
it follows that $r B_{\rho_\lambda, \delta}\left(\frac{x}{r}\right) = B_{\rho_{\lambda/r}, \delta}(x)$ for modular balls, for any $\delta > 0$. Hence, if $A \in \tau_{\lambda}$ for some $\lambda > 0$, then $rA \in \tau_{\lambda/r}$. The properties of the family $(\tau_\lambda)_{\lambda > 0}$, guarantee that if $A \in \tau_n$ for some $n \geq 1$, then there exists $m \geq 1$ such that $rA \in \tau_m$. Since $\rho$-balls are symmetric, this result extends to any $r \neq 0$. Finally, using the characterization of $\tau^{i}$-open subsets, it is readily seen that $rA$ is in $\tau^i$ for any $r \neq 0$, which completes the proof of Proposition \ref{tau^i-algebra}.
\end{proof}

The next proposition characterizes sequential convergence in the initial
topology $\tau^i$ in terms of the family $(\tau_\lambda)$.

\begin{proposition}\label{sequential-comparison}
For any sequence $(x_j)\subset X_\rho$, the following are equivalent:
\begin{enumerate}
\item[(i)] $x_j \overset{\tau^i}{\longrightarrow} x$;
\item[(ii)] $x_j \overset{\tau_\lambda}{\longrightarrow} x$ for every
$\lambda>0$.
\end{enumerate}
\end{proposition}
\begin{proof}
If $x_j\xrightarrow{\tau^i}x$, then
$x_j\xrightarrow{\tau_\lambda}x$ for every $\lambda$, because
$\tau_\lambda\subseteq\tau^i$.  Conversely, suppose convergence holds
in every $\tau_\lambda$.  A basic $\tau^i$-neighborhood of $x$ contains
a finite intersection $U_1\cap\cdots\cap U_N$, where
$U_k\in\tau_{\lambda_k}$ and $x\in U_k$.  The sequence is eventually in
each $U_k$, hence eventually in their intersection.  Therefore it
converges to $x$ in $\tau^i$.
\end{proof}

\noindent Proposition \ref{sequential-comparison} implies, in particular, that any \(\tau^i\)-convergent sequence has a unique limit within \(X_\rho\).\\

Next, we compare the Luxemburg norms associated with the modulars
$\rho_\lambda$, $\lambda>0$, and relate their induced topology to the
initial topology $\tau^i$.\\

\begin{proposition}\label{equivalence}
Consider the modular vector space $X_\rho$.
\begin{itemize}
\item[(i)] For every $\alpha>0$ and every $x\in X_\rho$,
\[
\|x\|_{\rho_\alpha}=\alpha\|x\|_\rho.
\]
\item[(ii)] For every $\alpha,\beta>0$, the Luxemburg norms associated
with $\rho_\alpha$ and $\rho_\beta$ are proportional. More precisely,
\[
\|x\|_{\rho_\alpha}
=
\frac{\alpha}{\beta}\|x\|_{\rho_\beta},
\qquad x\in X_\rho.
\]
Consequently, all these norms induce the same topology on $X_\rho$.\\
\item[(iii)] For any sequence $(x_j)\subset X_\rho$ and $x\in X_\rho$,
the following are equivalent:
\[
x_j\xrightarrow{\tau^i}x,
\qquad
x_j\xrightarrow{\rho_\lambda}x
\quad\text{for every }\lambda>0,
\qquad
\|x_j-x\|_\rho\longrightarrow0.
\]
\end{itemize}
\end{proposition}
\begin{proof}
For $\alpha>0$ and $x\in X_\rho$, by the definition of the Luxemburg
norm,
\[
\begin{aligned}
\|x\|_{\rho_\alpha}
&=
\inf\left\{\mu>0:
\rho_\alpha\left(\frac{x}{\mu}\right)\leq1\right\}\\
&=
\inf\left\{\mu>0:
\rho\left(\frac{\alpha x}{\mu}\right)\leq1\right\}\\
&= \|\alpha\ x\|_\rho = \alpha\|x\|_\rho.
\end{aligned}
\]
This proves $(i)$, and $(ii)$ follows immediately.\\
For $(iii)$, the equivalence between
$x_j\xrightarrow{\tau^i}x$ and
$x_j\xrightarrow{\rho_\lambda}x$ for every $\lambda>0$
follows from Proposition~\ref{sequential-comparison}. Thus it remains
to prove that
\[
x_j\xrightarrow{\rho_\lambda}x
\quad\text{for every }\lambda>0
\]
if and only if  $\|x_j-x\|_\rho\longrightarrow 0$.  By translation, it is enough to consider the case $x=0$.\\
Suppose first that $\|x_j\|_\rho\to0$. Fix $\lambda>0$. For all
sufficiently large $j$, one has
$\lambda\|x_j\|_\rho\leq1$. By Proposition~\ref{standard},
\[
\rho(\lambda x_j)
\leq
\|\lambda x_j\|_\rho
=
\lambda\|x_j\|_\rho,
\]
and hence
\[
\rho_\lambda(x_j)=\rho(\lambda x_j)\longrightarrow0.
\]
Thus $x_j\xrightarrow{\rho_\lambda}0$ for every $\lambda>0$.\\
Conversely, suppose that
$x_j\xrightarrow{\rho_\lambda}0$ for every $\lambda>0$, but that
$\|x_j\|_\rho$ does not converge to $0$. Passing to a subsequence if
necessary, there exists $\varepsilon_0>0$ such that
\[
\|x_j\|_\rho>\varepsilon_0
\qquad\text{for all }j.
\]
By the definition of the Luxemburg norm,
\[
\rho\left(\frac{x_j}{\varepsilon_0}\right)>1
\qquad\text{for all }j.
\]
Equivalently,
\[
\rho_{1/\varepsilon_0}(x_j)>1
\qquad\text{for all }j,
\]
which contradicts
$x_j\xrightarrow{\rho_{1/\varepsilon_0}}0$.
Therefore $\|x_j\|_\rho\to0$, and the proof is complete.
\end{proof}

\smallskip

\begin{theorem}\label{norminitial}
Let $\tau_*$ be the collection of all subsets $A\subset X_\rho$ with the
following property: for every $x\in A$, there exist $\lambda>0$ and
$\varepsilon>0$, both possibly depending on $x$, such that $B_{\rho_\lambda,\varepsilon}(x)\subseteq A$.  Then the following hold:
\begin{itemize}
\item[$(i)$] $\tau_*$ is a topology on $X_\rho$ and
\[
\tau_\lambda\subseteq\tau_*
\qquad\text{for every }\lambda>0.
\]
Consequently, since $\tau^i$ is the weakest topology stronger than every
$\tau_\lambda$,
\[
\tau^i\subseteq\tau_*.
\]
\item[$(ii)$] For any sequence $(x_k)\subset X_\rho$ and $x\in X_\rho$,
\[
x_k\xrightarrow{\tau_*}x
\quad\Longleftrightarrow\quad
x_k\xrightarrow{\rho_\lambda}x
\quad\text{for every }\lambda>0.
\]
Consequently, $\tau^i$ and $\tau_*$ have the same convergent sequences (but they are not necessarily equal).
\item[$(iii)$] Let $\tau_{\|\cdot\|_\rho}$ denote the topology induced on
$X_\rho$ by the Luxemburg norm. Then
\[
\tau_{\|\cdot\|_\rho}=\tau_*.
\]
Moreover, $\tau_{\|\cdot\|_\rho}$ is the weakest first-countable topology
that is stronger than every $\tau_\lambda$, $\lambda>0$.
\end{itemize}
\end{theorem}
\begin{proof}
For $(i)$, it is straightforward to show that $\tau_*$ is a topology. Indeed, if $A \in \tau_*$, $B \in \tau_*$ and $x\in B\cap A$, then there must exist $\lambda_1, \lambda_2, \varepsilon_1, \varepsilon_2 \in (0,\infty)$ such that $B_{\rho_{\lambda_1}, \varepsilon_1}(x) \subseteq A$ and $B_{\rho_{\lambda_2}, \varepsilon_2}(x) \subseteq B$. Because of convexity, $\rho(\lambda_i(x-y))\leq \rho\left(\max\{\lambda_1,\lambda_2\}(x-y)\right)$; this implies the inclusions $B_{\rho_{\max\{\lambda_1,\lambda_2\}}, \min\{\varepsilon_1, \varepsilon_2\}}(x)\subset B_{\rho_{\lambda_1}, \varepsilon_1}(x) \cap B_{\rho_{\lambda_2}, \varepsilon_2}(x) \subseteq A \cap B$. \\
The verification of the fact that $\tau_*$ is closed under arbitrary unions is straightforward.\\
It is evident by definition that for any $\lambda > 0$, $\tau_{\lambda} \subseteq \tau_*$.\\
To prove $(ii)$, assume that $(x_k) \overset{\tau_{\ast}}{\rightarrow} x$. Suppose there exists $\lambda_0 > 0$ such that $\rho_{\lambda_0}(x - x_k) \not\rightarrow 0$ as $k \rightarrow \infty$. Then there exists $\delta > 0$ and a subsequence $(x_{k_j})$ such that $\rho_{\lambda_0}(x - x_{k_j}) \geq \delta$ for all $j \in \mathbb{N}$. Set $S = \{x_{k_j}; j \in \mathbb{N}\}$. Clearly, $x \notin S$.\\
The subsequence $(x_{k_j})$ either contains a $\rho_{\lambda_0}$-convergent subsequence or no subsequence of $(x_{k_j})$ $\lambda_0$-converges. In the first case, select a subsequence, say $(y_i)$, such that $\rho_{\lambda_0}(y_i - y) \rightarrow 0$ as $i \rightarrow \infty$. Set $B = \{y_i; i \in \mathbb{N}\} \cup \{y\}$. Necessarily, $x \notin B$. The set $B$ is $\rho_{\lambda_0}$-closed in $X_{\rho}$ and therefore its complement $X_{\rho} \setminus B$ is $\rho_{\lambda_0}$-open, hence $\tau_{\ast}$-open, and it contains $x$. But $(y_i)$ is a subsequence of $(x_k)$, which $\tau_{\ast}$-converges to $x$. This is clearly a contradiction.\\
Similarly, if no subsequence of $(x_{k_j})$ is $\rho_{\lambda_0}$-convergent, then $S$ is $\rho_{\lambda_0}$-closed, hence $\tau_{\ast}$-closed, and $S$ does not contain $x$. Again, a contradiction is reached by observing that in this case, $X_{\rho} \setminus S$ is a $\tau_{\ast}$-open set containing $x$.\\
Conversely, assume $(x_k) \overset{\rho_{\lambda}}{\rightarrow} x$ for all $\lambda > 0$. Let $V$ be a neighborhood of $0$ in $\tau_{\ast}$. By definition, there exist $\delta > 0$, $\varepsilon > 0$ such that the modular ball $B_{\rho_{\delta}, \varepsilon}(0) \subset V$. Since $\rho_{\delta}(x - x_j) \rightarrow 0$ as $j \rightarrow \infty$ it is immediate that $x - x_j \in B_{\rho_{\delta}, \varepsilon}(0)$ for large enough $j$. Thus, $x_k \overset{\tau_{\ast}}{\rightarrow} x$, as claimed.\\
The inclusion $\tau_{\|\cdot\|_{\rho}} \subseteq \tau_{\ast}$ is tackled next. To this end, let $A \subseteq X_{\rho}$ be $\|\cdot\|_{\rho}$-open and $x \in A$. Then there exists $\varepsilon > 0$ such that $\{y \in X_\rho : \|y - x\|_{\rho} < \varepsilon\} \subset A$. By the definition of the Luxemburg norm, if $\rho_{\frac{2}{\varepsilon}}(x - y) < 1$, then $\|y - x\|_{\rho} < \varepsilon$. Hence, for any $x \in A$, the modular ball $B_{\rho_{2/\varepsilon}, 1}(x) \subset A$. It follows that $A$ is $\tau_{\ast}$-open.\\
On the other hand, let $V$ be a $\tau_{\ast}$-open set. Take $x \in V$. By definition, there exist $\lambda_x > 0$ and $\varepsilon_x < 1$ such that 
$$
\rho(\lambda_x(y - x)) < \varepsilon_x \Rightarrow y \in V.
$$
If $\|\lambda_x(x - y)\|_{\rho} < \varepsilon_x < 1$, then one has, by virtue of Proposition \ref{standard} (iii), that
\begin{align*}
\rho(\lambda_x(y - x)) &= \rho\Big(\|\lambda_x(x - y)\|_{\rho}\|\lambda_x(x - y)\|^{-1}_{\rho}\lambda_x(y - x)\Big) \\
&\leq \|\lambda_x(x - y)\|_{\rho} \rho\Big(\|\lambda_x(x - y)\|^{-1}_{\rho}\lambda_x(x - y)\Big) \\
&< \varepsilon_x.
\end{align*}
It follows that the norm ball $\{y : \|x - y\|_{\rho} < \varepsilon_x \lambda_x^{-1}\}$ is contained in $V$, and thus $V$ is open in $\tau_{\|\cdot\|_{\rho}}$. Hence, $\tau_{\ast} \subseteq \tau_{\|\cdot\|_{\rho}}$, as claimed.\\
Finally, let $\tau$ be a first-countable topology on $X_\rho$ stronger
than every $\tau_\lambda$.  If $x_j\xrightarrow{\tau}x$, then
$x_j\xrightarrow{\tau_\lambda}x$ for every $\lambda$, and Proposition
\ref{equivalence} gives $\|x_j-x\|_\rho\to0$.  Thus the identity map
\[
I:(X_\rho,\tau)\longrightarrow
(X_\rho,\tau_{\|\cdot\|_\rho})
\]
is sequentially continuous.  A sequentially continuous map from a
first-countable space is continuous: if continuity failed at $x$, take a
decreasing countable neighborhood base $(U_n)$ at $x$ and choose
$x_n\in U_n$ whose images avoid a fixed neighborhood of $I(x)$; then
$x_n\to x$, a contradiction.  Hence $I$ is continuous and
$\tau_{\|\cdot\|_\rho}\subseteq\tau$, proving the asserted minimality.

\end{proof}

\begin{remark}
{\normalfont Theorem~\ref{norminitial} asserts equality of convergent
sequences between $\tau^i$ and $\tau_*$, not equality of the two
topologies. The topology $\tau_*$ is the first-countable envelope;
$\tau^i$ need not itself be first-countable or sequential.}
\end{remark}

\section{Compatibility with the algebraic structure}\label{compatibility}

The question arises whether the modular topology is compatible with the algebraic structure of \(X_{\rho}\). The central results of this section are Theorems \ref{Mainequivalence} and \ref{delta2-rcontinuity}, which to the authors' best knowledge are new in the literature.\\
  It will be shown in particular that the validity of the \(\Delta_2\)-condition is equivalent to the compatibility of the modular topology with the linear structure of the underlying vector space; that is, the modular topology is a TVS topology iff the modular satisfies the $\Delta_2$-condition. The results in this section should be compared with Chapter 1, $\S$ 6 in \cite{M:1983}, which deals with a very different question, namely the compatibility of the vector space structure and the topology generated by the modular balls as neighborhoods of the origin, which are {\it not} open in the modular topology considered here. Thus, Theorem \ref{Mainequivalence} answers a question that is not addressed by, nor can be derived from Theorem 6.2 in \cite{M:1983}. \\ The following well known Proposition (see \cite{M:1983}) will be tacitly used in the proof of Theorem \ref{Mainequivalence}. We include it for the sake of clarity.

\begin{proposition}\label{ModCV-Delta2}
If a modular $\rho$ on a vector space satisfies the $\Delta_2$-condition, then for any sequence $(x_j)\subset X$ the following statements are equivalent:
\begin{itemize}
\item [$(i)$] $x_j\overset {\rho}\rightarrow x \in X_{\rho}$\\
\item [$(ii)$] $cx_j\overset {\rho}\rightarrow cx \in X_{\rho}$ for any constant $c$.
\end{itemize}
\end{proposition}

\begin{proof}\leavevmode
The claim follows by observing that the validity of the $\Delta_2$-condition implies that if $\rho((x_j-x))\rightarrow 0$ as $j\rightarrow \infty$ and $c$ is any constant with $|c|<2^M$ for some $M$ then $\rho(c(x_j-x))\leq \rho(2^M(x_j-x))\rightarrow 0$ as $j\rightarrow \infty$.
\end{proof}

Proposition \ref{ModCV-Delta2} suggests that the $\Delta_2$-condition is the key ingredient ensuring compatibility between the modular topology and the underlying vector space structure. The next theorem confirms this observation by showing that the $\Delta_2$-condition is equivalent to several fundamental topological and analytical properties.

\begin{theorem}\label{Mainequivalence}
Let $\rho$ be a convex modular on a real vector space $X$, and let $\tau_{\rho}$ and $\tau_{\|\cdot\|_{\rho}}$ be the modular topology and the norm topology, respectively. Then, the following conditions are equivalent:
\begin{enumerate}
\item[(i)] $\tau_{\rho}$ is a TVS topology on $X_{\rho}$.
\item[(ii)] $\rho$ satisfies the \(\Delta_2\)-condition.
\item[(iii)] $\|\cdot\|_{\rho}$-convergence is equivalent to $\rho$-convergence.
\item[(iv)] $\tau_{\|\cdot\|_{\rho}} = \tau_{\rho}$.
\item[(v)] $(X_{\rho}, \tau_{\rho})$ is normable.
\end{enumerate}
\end{theorem}

\begin{proof}\leavevmode
\begin{enumerate}
\item[] $(i) \Rightarrow (ii)$. Assume $(i)$. Fix a modular neighborhood of $0$, say $\mathcal{N}$, and let $\mathcal{M}$ be another neighborhood of $0$ such that $\mathcal{M} + \mathcal{M} \subset \mathcal{N}$; this is possible since it is assumed that the algebraic operations on $X$ are continuous with respect to $\tau_{\rho}$. For the same reason, there must exist $r > 0$ such that $B_{\rho, r}(0) \subseteq \mathcal{M}$, and thus
\[
2B_{\rho, r}(0) \subset B_{\rho, r}(0) + B_{\rho, r}(0) \subset \mathcal{N}.
\]
Pick an arbitrary sequence $(x_j)$ such that $(x_j) \overset{\rho}{\rightarrow} 0$. Then, for some $n_0 \geq 1$, we have $x_j \in B_{\rho, r}(0)$ for $j \geq n_0$, which guarantees $2x_j \in \mathcal{N}$ for $j \geq n_0$. On account of the arbitrariness of $\mathcal{N}$, it follows that $(2x_j)\overset{\tau_{\rho}}{\rightarrow} 0$. By Theorem \ref{convergencecharacterization}, $(2x_j)\overset{\rho}{\rightarrow} 0$. Hence, $\rho$ satisfies the \(\Delta_2\)-condition, as claimed.

\item[] $(ii) \Rightarrow (iii)$.
Proposition~\ref{ModCV-Delta2} shows that $\rho(\lambda(x_j-x))\to0$ for every $\lambda>0$. Proposition~\ref{equivalence} then gives $\|x_j-x\|_\rho\to0$. Conversely, norm convergence always implies modular convergence by Proposition~\ref{standard} $(iii)$.

\item[] $(iii) \Rightarrow (iv)$. Every $\tau_\rho$-open set is norm-open: if $B_{\rho,\varepsilon}(x)\subset A$, then the norm ball of radius $\min\{\varepsilon,1/2\}$ about $x$ lies in $A$ by Proposition~\ref{standard} $(iii)$. Thus $\tau_\rho\subseteq\tau_{\|\cdot\|_\rho}$.  Conversely, if $A$ is norm-closed and $(x_j)\subset A$ satisfies $x_j\xrightarrow{\rho}x$, condition $(iii)$ gives norm convergence and hence $x\in A$.  Proposition~\ref{rho-topology} shows that $A$ is $\tau_\rho$-closed, proving the reverse inclusion.

\item[] $(iv) \Rightarrow (v)$. This is immediate.

\item[] $(v) \Rightarrow (i)$. If $(v)$ holds, then $\tau_{\rho}$ is the topology generated by a norm $\|\cdot\|$, which is a TVS topology on $X_{\rho}$.
\end{enumerate}
Thus, the proof of Theorem \ref{Mainequivalence} is complete.
\end{proof}

\smallskip

The following observation, whose proof is a straightforward application of the Heine-Borel property, highlights that the phenomena considered in this
paper are inherently infinite-dimensional.

\begin{proposition}\label{equivalence-finte-dim}
Let $\rho:X\to[0,\infty]$ be a convex modular on a finite-dimensional
vector space $X$. Then $\rho$ satisfies the $\Delta_2$-condition.
Consequently, the modular topology and the Luxemburg norm topology
coincide on $X_{\rho}$.
\end{proposition}

\smallskip

\begin{theorem}\label{equivalence-compact}
Let $\rho:X\rightarrow [0,\infty]$ be a convex modular on a  vector space $X$ and $K\subset X_\rho$ be compact relative to the $\|\cdot\|_{\rho}$-topology. Then any sequence $(x_j)\subseteq K$ is $\rho$-convergent if and only if it is $\|\cdot\|_{\rho}$-convergent.
\end{theorem}
\begin{proof} The proof is a standard application of compactness.
\end{proof}

Perhaps the most significant shortcoming of the topology $\tau_{\rho}$ is the counterintuitive fact that modular balls are not necessarily open. The following theorem sheds some light on this issue.

\begin{theorem}\label{delta2-rcontinuity}
Let $\rho$ be a convex, left-continuous modular on a vector space $X$ and consider the following statements:
\begin{enumerate}
\item[(i)] each open $\rho$-ball is $\rho$-open;
\item[(ii)] $\rho$ is right-continuous;
\item[(iii)] each open $\rho$-ball is norm-open.
\end{enumerate}
Then $(i)\Rightarrow (ii)\Rightarrow (iii)$.
In addition, if $\rho$ satisfies the \(\Delta_2\)-condition, then $(iii)\Rightarrow (i)$.
\end{theorem}

\begin{proof}\leavevmode
\begin{enumerate}
\item[(1)]  $(i) \Rightarrow (ii)$. Assume $(i)$. If $(ii)$ fails, let $x_0\in X_\rho$ be such that for some sequence $(\lambda_j)\searrow 1 $ and $\delta > 0$, we have $\rho(\lambda_j x_0)\geq \rho(x_0)+\delta$ for any $j \in \mathbb{N}$. Consider the open modular ball $B=\{z\in X_\rho:\ \rho(z)<\rho(x_0)+\frac{\delta}{2}\}$. Since $(\lambda_j x_0)\overset{\rho}{\rightarrow}x_0$ and $x_0 \in B$, it is easy to see that $B$ does not contain any $\rho$-ball centered at $x_0$. Hence $B$ is not $\rho$-open, contradicting $(i)$.
\item[(2)]  $(ii) \Rightarrow (iii)$. Assume $(ii)$ holds. First, establish that for $r>0$ 
\begin{equation}\label{1tor}
B_{\rho,r}(0)=\{x \in X_\rho:\ \rho(x)<r\}=\{x \in X_\rho:\ \mu_{B_r}(x)<1\},
\end{equation}
where $\mu_{B_r}$ is the Minkowski functional associated with $B_r(0)$. If $\rho(x)<r$, then by right-continuity, there exists $\lambda>1$ such that $\rho(\lambda x)<r$, i.e., $\mu_{B_r}(x)\leq \lambda^{-1}<1$. Conversely, if 
$$\mu_{B_r}(x)=\inf\{\gamma>0: \rho\left(\gamma^{-1}x\right)\leq r \}<1,$$
then for some $\gamma \in (0,1)$, one has $\rho (\gamma^{-1}x)\leq r$, yielding
$$\rho(x)=\rho(\gamma^{-1}\gamma x)\leq \gamma \rho(\gamma^{-1}x)\leq \gamma r<r.$$
Thus, (\ref{1tor}) holds. Since the functional $\mu_{B_r}$ is a norm and it is equivalent to the Luxemburg norm, the right-hand side in (\ref{1tor}) is $\|\cdot\|_{\rho}$-open, and $(iii)$ follows.
\end{enumerate}
Finally, if the $\Delta_2$-condition holds, then the norm topology and the $\rho$-topology coincide. If that is the case, it is clear that $(iii)\Rightarrow (i)$.
\end{proof}
\begin{remark}{\normalfont In the light of the preceding theorem, if the modular $\rho$ fails to be right-continuous, modular balls $B_{\rho, \delta}(x)=\{y:\rho(x-y)<\delta\}$ are not expected to be open. In particular, modular balls cannot be expected to be open in the variable exponent spaces $\ell^{(p_n)}, L^{p(\cdot)}$ when the exponent is unbounded. More strongly yet, Corollary \ref{emptyinterior} in the next section will establish not only that modular balls fail to be modularly open in this case, but that they have empty interior.  }

\end{remark}
A straightforward consequence of Theorem \ref{convergencecharacterization} is that the continuous functions with respect to the modular topology are precisely those that are sequentially continuous.\\
The next result follows from Corollary \ref{l*convergence} and \cite[Theorem 2.2]{dudley}; we opted for an ad-hoc proof based on the modular structure of $X$, for the sake of keeping the exposition self contained.

\begin{theorem}\label{continuitysequential}
Let $(X,\rho_X)$ and $(Y,\rho_Y)$ be modular spaces with modular topologies $\tau_{\rho_X}$ and $\tau_{\rho_Y}$, respectively. Then
$$f:(X_{\rho_X},\tau_{\rho_X})\rightarrow (Y_{\rho_Y},\tau_{\rho_Y})$$
is continuous if and only if it is sequentially continuous.
\end{theorem}
\begin{proof}
Continuity always implies sequential continuity. Conversely, suppose
$f$ is sequentially continuous and let $C\subset Y_{\rho_Y}$ be closed.
If $(x_j)\subset f^{-1}(C)$ and $x_j\xrightarrow{\tau_{\rho_X}}x$,
then Theorem~\ref{convergencecharacterization} and sequential continuity
give $f(x_j)\xrightarrow{\tau_{\rho_Y}}f(x)$. Hence $f(x)\in C$.
Thus $f^{-1}(C)$ is sequentially closed. By the definition of the
modular topology (Proposition~\ref{rho-topology}), it is closed, proving
that $f$ is continuous.
\end{proof}

\section{The modular topology of $\ell^{(p_n)}$ and $L^{p(\cdot)}$}\label{lpn}
The abstract theory developed in the previous sections becomes particularly transparent in variable exponent spaces, where the failure of the $\Delta_2$-condition gives rise to genuinely new topological phenomena.

In this section the modular topologies inherent to the variable exponent sequence spaces and Lebesgue spaces (Definitions \ref{deflp} and \ref{Lp}) are studied in detail in the light of the results in Sections \ref{modulartopologies}, \ref{lambda} and \ref{compatibility}. We focus exclusively on the class of unbounded exponents, since in the bounded case the modular topology coincides with the norm topology, which has been widely studied in the literature, see \cite{DHHR, KR, ML} and the references therein.\\
In particular, we prove that the boundedness of the exponent $p$ is equivalent to the right-continuity of the modular $\rho_p$ and the openness of the modular balls in the modular topology (Theorems \ref{r-continuous-p^+}, \ref{r-continuous-p^+L} and their Corollaries).
Furthermore, we show that in the unbounded setting, modular balls have empty interior and modularly open sets are unbounded and that the modular topology is not first-countable (hence, not second countable).
Finally, we prove in Theorem \ref{separability} that, in contrast to the norm topologies, the modular topologies of $\ell^{(p_n)}$ and $L^{p(\cdot)}$ are separable, even for unbounded $p$. 

\begin{theorem}\label{r-continuous-p^+}
Let \( \mathbf{p}: = (p_j) \subset (1,\infty) \) be a sequence and  \(\rho_{\mathbf{p}}: \ell^{(p_j)} \to [0,\infty]\) be as in Definition \ref{deflp}.
Then the following are equivalent:
\begin{enumerate}
\item[(i)] $\rho_{\mathbf{p}}$ is right-continuous on $\ell^{(p_j)}$,
\item[(ii)] $p^+:=\sup\limits_{j\in {\mathbb N}}\ p_j<\infty$,
\item[(iii)] $\rho_{\mathbf{p}}$ satisfies the \(\Delta_2\)-condition.
\end{enumerate}
\end{theorem}
\begin{proof}
We will first show that if $p^+ = \infty$, then $\rho_{\mathbf{p}}$ is not right-continuous. There exists a strictly increasing sequence \((n_k)\subseteq \mathbb{N}\), with \(n_k\geq k\), such that \(p_{n_k}>k^2\) and that \((p_{n_k})\) is strictly increasing. Define the sequence \((a_k)\) by setting \(a_{n_k} = p_{n_k}^{-\frac{1}{p_{n_k}}}\) for \(k=1,2,\ldots\) and \(a_j = 0\) otherwise. Then \((a_j) \in \ell^{(p_n)}\), as 
$$\rho_{\mathbf{p}}((a_j)) = \textstyle \sum\limits_{k=1}^{\infty} \frac{1}{p_{n_k}} \leq \textstyle \sum\limits_{k=1}^{\infty} \frac{1}{k^2} < \infty.$$
On the other hand, for any \(\lambda > 1\), we have
$$\rho_{\mathbf{p}}((\lambda a_j)) = \textstyle \sum\limits_{k=1}^{\infty} \frac{\lambda^{p_{n_k}}}{p_{n_k}} = \infty.$$
It follows that \(\rho_{\mathbf{p}}\) is not right-continuous, thus \((i) \Rightarrow (ii)\).\\
Conversely, assume \((ii)\) holds, i.e., \(p^+ < \infty\). For any \((a_j) \in \ell^{(p_n)}\), the following inequality holds:
$$\rho_{\mathbf{p}}((2a_j)) \leq 2^{p^+} \rho_{\mathbf{p}}((a_j)).$$
If \(\rho_{\mathbf{p}}((a_j)) = \infty\), then for any \(\lambda > 1\), it is clear that
$$\rho_{\mathbf{p}}(\lambda (a_j)) \geq \rho_{\mathbf{p}}((a_j)) = \infty.$$
Furthermore, it is clear that \(\lim_{\lambda \to 1^+} \rho_{\mathbf{p}}(\lambda (a_j)) = \rho_{\mathbf{p}}((a_j))\). If \(\rho_{\mathbf{p}}((a_j)) < \infty\) and \(\lambda_k \searrow 1\), then, for some \(\delta > 0\) and for each \(k\), \((\lambda_k)^{p_j} < (1 + \delta)^{p^+}\). Thus, for any \(j \geq 1\),
$$|\lambda_k a_j|^{p_j} \leq (1 + \delta)^{p^+} |a_j|^{p_j},$$
and the series is summable. On account of Lebesgue's dominated convergence theorem it follows that
$$\lim_{\lambda \to 1^+} \textstyle \sum\limits_{j=1}^{\infty} |\lambda_k a_j|^{p_j} = \textstyle \sum\limits_{j=1}^{\infty} |a_j|^{p_j}.$$
Thus, \((i) \iff (ii)\). The implication \((ii) \Rightarrow (iii)\) follows directly. Finally, if \((ii)\) is not satisfied, then \((iii)\) cannot hold. If the sequence \((p_n)\) is unbounded, a strictly increasing subsequence \((n_k)\) of natural numbers can be chosen such that \(p_{n_k} > k\) and \((p_{n_k})\) is strictly increasing. Let \((x_k)\) be the sequence equal to \(\frac{1}{2}\) on each \(n_k\) and zero otherwise. Then the sequence \(\Big((x_k)\mathbbm{1}_{\{i:\ i > n_m\}}\Big)_{m \geq 1}\) \(\rho_{\mathbf{p}}\)-converges to \(0\), but \(\Big(2 (x_k)\mathbbm{1}_{\{i:\ i > n_m\}}\Big)_{m \geq 1}\) does not. Hence, \((iii)\) fails, proving that \((iii) \Rightarrow (ii)\).
\end{proof}
A similar situation arises in the continuous case:

\begin{theorem}\label{r-continuous-p^+L}
Let $\Omega \subseteq {\mathbb R}^n$ be bounded, $L^{p(\cdot)}(\Omega)$ and $\rho_{{p}}$ be as in Definition \ref{Lp}.
Then the following are equivalent:
\begin{enumerate}
\item[(i)] $\rho_{{p}}$ is right-continuous,
\item[(ii)] $p^+ < \infty$,
\item[(iii)] $\rho_{{p}}$ satisfies the \(\Delta_2\)-condition.
\end{enumerate}
\end{theorem}
\begin{proof}
For each $k\in\mathbb N$, set
\[\Omega_k=\{x\in\Omega:\ k\leq p(x)<k+1\}.\]
Assume first that $p^+=\infty$. Without loss of generality it can be assumed that there exists an increasing sequence
$(k_j)$ tending to infinity such that $|\Omega_{k_j}|>0$ for every $j$.  Define
\[u=\sum_{j=1}^{\infty}
{\mathbbm 1}_{\Omega_{k_j}}\,
2^{-1}|\Omega_{k_j}|^{-1/p(x)}.\]
Then
\[\rho_p(u)= \sum_{j=1}^{\infty} \frac{1}{|\Omega_{k_j}|}
\int_{\Omega_{k_j}}2^{-p(x)}\,dx \leq
\sum_{j=1}^{\infty}2^{-k_j}<\infty,\]
whereas
\[\rho_p(2u) = \sum_{j=1}^{\infty} \frac{1}{|\Omega_{k_j}|}
\int_{\Omega_{k_j}}1\,dx = \sum_{j=1}^{\infty}1 
=
\infty.
\]
For $N\ge1$ let
\[
u_N=\sum_{j=N}^{\infty}{\mathbbm 1}_{\Omega_{k_j}}
2^{-1}|\Omega_{k_j}|^{-1/p(x)}.
\]
Then $\rho_p(u_N)\le\sum_{j=N}^{\infty}2^{-k_j}\to0$, whereas
$\rho_p(2u_N)=\infty$ for every $N$.  This is precisely the
sequential failure of $\Delta_2$.
Thus, $\rho_p$ does not satisfy the \(\Delta_2\)-condition, proving
$(iii)\Rightarrow(ii)$.\\
Next, consider the function
\[
v(x)=\sum_{j=1}^{\infty}
{\mathbbm 1}_{\Omega_{k_j}}(x)\,
k_j^{-2/k_j}|\Omega_{k_j}|^{-1/p(x)}.
\]
Since $p(x)\geq k_j$ on $\Omega_{k_j}$,
\[
\rho_p(v)
=
\sum_{j=1}^{\infty}
\frac{1}{|\Omega_{k_j}|}
\int_{\Omega_{k_j}}
k_j^{-2p(x)/k_j}\,dx
\leq
\sum_{j=1}^{\infty}k_j^{-2}
<\infty.
\]
On the other hand, for every $\theta>1$, using
$k_j\leq p(x)<k_j+1$ on $\Omega_{k_j}$, we obtain
\[
\begin{aligned}
\rho_p(\theta v)
&=
\sum_{j=1}^{\infty}
\frac{1}{|\Omega_{k_j}|}
\int_{\Omega_{k_j}}
\theta^{p(x)}k_j^{-2p(x)/k_j}\,dx\\
&\geq
\sum_{j=1}^{\infty}
\theta^{k_j}k_j^{-2(k_j+1)/k_j}.
\end{aligned}
\]
Indeed, for every $\theta>1$,
\[
\theta^{k_j}k_j^{-2(k_j+1)/k_j}
\longrightarrow\infty
\qquad\text{as }j\longrightarrow\infty.
\]
Hence $\rho_p(\theta v)=\infty$ for every $\theta>1$.
Since $\rho_p(v)<\infty$, this shows that $\rho_p$ is not
right-continuous at $v$. Therefore $(i)\Rightarrow(ii)$.  \\
Conversely, assume $p^+<\infty$. Then
\[
\rho_p(2f)
=
\int_\Omega 2^{p(x)}|f(x)|^{p(x)}\,dx
\leq
2^{p^+}\rho_p(f),
\]
so $\rho_p$ satisfies the \(\Delta_2\)-condition. Thus
$(ii)\Rightarrow(iii)$.\\
Finally, assume $p^+<\infty$. For $\lambda>1$ sufficiently close to $1$,
we have
\[
|\lambda f(x)|^{p(x)}
\leq
(1+\delta)^{p^+}|f(x)|^{p(x)}
\]
for some $\delta>0$. Since the right-hand side is integrable whenever
$\rho_p(f)<\infty$, the dominated convergence theorem yields
\[
\lim_{\lambda\rightarrow 1^+}\rho_p(\lambda f)=\rho_p(f).
\]
Hence $\rho_p$ is right-continuous, and $(ii)\Rightarrow(i)$.

Therefore, $(i)$, $(ii)$, and $(iii)$ are equivalent.
\end{proof}

\smallskip

The following corollary is a direct consequence of Theorems
\ref{delta2-rcontinuity}, \ref{r-continuous-p^+}, and
\ref{r-continuous-p^+L}.

\begin{corollary}
In the notation of Theorems \ref{r-continuous-p^+} and
\ref{r-continuous-p^+L}, all modular balls in $\ell^{(p_j)}$ and
$L^{p(\cdot)}(\Omega)$ are open in the corresponding modular topology
if and only if the exponent is bounded, i.e., $p^+<\infty$.
\end{corollary}
\begin{proof}
If $p^+<\infty$, the corresponding modular satisfies $\Delta_2$; its
modular topology is therefore the Luxemburg norm topology by
Theorem~\ref{Mainequivalence}, and its balls are open by
Theorem~\ref{delta2-rcontinuity}. Conversely, if all modular balls are
open, that theorem gives radial right-continuity, and
Theorems~\ref{r-continuous-p^+} and~\ref{r-continuous-p^+L} imply
$p^+<\infty$.
\end{proof}
The preceding corollary implies that, when $p^+=\infty$, not all modular
balls are open. The following result is stronger: in the sequence-space
case, if $(p_n)$ is unbounded, then no modular ball is open in the
modular topology.\\

\begin{theorem}\label{modularballsnotopen}
Let $(p_n)$ be unbounded, i.e., $p^+=\infty$. Then no modular ball in
$\ell^{(p_n)}$ is $\tau_{\rho_{\mathbf p}}$-open.
\end{theorem}
\begin{proof}
Construct a sequence $(n_k)$ as follows. Choose $n_1$ such that
$p_{n_1}\geq 1$, and, inductively, choose $n_k$ so that
\[
p_{n_k}\geq \max\{k,p_{n_{k-1}}\}, \qquad k>1.
\]
Let $\mathbf{s} = {\mathbbm 1}_{\{n_1,n_2,\ldots\}}$.  Then $\mathbf{s}\in\ell^{(p_n)}$, since
\[
\rho_p\left(2^{-1}\mathbf{s}\right)
=
\sum_{k=1}^{\infty}2^{-p_{n_k}}
\leq
\sum_{k=1}^{\infty}2^{-k}
<\infty.
\]
Fix $\delta>0$. For sufficiently small $\varepsilon>0$,  $(1-\varepsilon)\mathbf{s} \in B_{\rho_p,\delta}(\mathbf{s})$, since
\[
\rho_p\bigl((1-\varepsilon)\mathbf{s}-\mathbf{s}\bigr)
=
\rho_p(-\varepsilon\mathbf{s})
=
\sum_{k=1}^{\infty}\varepsilon^{p_{n_k}}
<\delta.
\]
Moreover,  $\rho_p\bigl((1-\varepsilon)\mathbf{s}\bigr)<\infty$. Now $(1-\varepsilon)\mathbf{s}$ can be approximated in the modular sense
by its truncations
\[
\mathbf{s}^{(N)}
=
(1-\varepsilon)(s_1,\ldots,s_N,0,0,\ldots).
\]
Indeed,
\[
\rho_p\bigl(\mathbf{s}^{(N)}-(1-\varepsilon)\mathbf{s}\bigr)
\longrightarrow 0
\qquad\text{as }N\longrightarrow\infty.
\]
On the other hand, for every $N$,
\[
\rho_p\bigl(\mathbf{s}^{(N)}-\mathbf{s}\bigr)=\infty,
\]
because $\mathbf{s}^{(N)}-\mathbf{s}$ has infinitely many coordinates
equal to $-1$. Hence  $\mathbf{s}^{(N)}\notin B_{\rho_p,\delta}(\mathbf{s})$.  
It follows that $(1-\varepsilon)\mathbf{s}$ is not an interior point of
$B_{\rho_p,\delta}(\mathbf{s})$, and therefore
$B_{\rho_p,\delta}(\mathbf{s})$ is not $\tau_{\rho_p}$-open.\\
More generally, the same argument shows that no
$\mathbf{f}\in B_{\rho_p,\delta}(\mathbf{s})$ with
$\rho_p(\mathbf{f})<\infty$ is an interior point of
$B_{\rho_p,\delta}(\mathbf{s})$: every $\tau_{\rho_p}$-neighborhood of
$\mathbf{f}$ contains sufficiently long truncations
\[
\mathbf{f}^{(N)}=(f_1,\ldots,f_N,0,0,\ldots),
\]
whereas, for sufficiently large $N$,  $\mathbf{f}^{(N)}\notin B_{\rho_p,\delta}(\mathbf{s})$.  Finally, for any $\mathbf{t}\in\ell^{(p_n)}$ and $\delta>0$,
\[
B_{\rho_p,\delta}(\mathbf{t})
=
\mathbf{t}-\mathbf{s}
+
B_{\rho_p,\delta}(\mathbf{s}).
\]
By the translation invariance of $\tau_{\rho_p}$ established in
Proposition~\ref{basic-properties}(1), no modular ball is open in the
modular topology of $\ell^{(p_n)}$.
\end{proof}

\smallskip

\begin{theorem}\label{general-unbounded-example}
Let $\Omega\subset\mathbb{R}^n$ be an arbitrary domain and let
$p:\Omega\to[1,\infty)$ be measurable with $p^+=\infty$.
Then there exists $v\in L^{p(\cdot)}(\Omega)$ such that
\[
\rho_p(v)=\infty
\]
and
\[
\rho_p(\lambda v)<\infty
\qquad\text{for every }0<\lambda<1.
\]
Moreover, for every $\delta>0$, the modular ball
$B_{\rho_p,\delta}(v)$ is not $\tau_{\rho_p}$-open. Consequently,
no modular ball in $L^{p(\cdot)}(\Omega)$ is $\tau_{\rho_p}$-open.
\end{theorem}

\begin{proof}

Since $p^+=\infty$, there exists a strictly increasing sequence
$(k_j)$ of positive integers, with $k_j\to\infty$, such that the sets
\[
F_j:=\{x\in\Omega:k_j\leq p(x)<k_j+1\}
\]
have positive measure. Choose measurable sets
$E_j\subset F_j$ with $0<|E_j|<\infty$.  These sets are pairwise
disjoint. Passing to a subsequence if necessary, we may
assume that $k_j\geq j$ for every $j$.
Define
\[
v(x)
=
\sum_{j=1}^{\infty}
k_j^{1/k_j}
|E_j|^{-1/p(x)}
{\mathbbm 1}_{E_j}(x).
\]
Since the sets $E_j$ are pairwise disjoint,
\[
\rho_p(v)
=
\sum_{j=1}^{\infty}
\frac{1}{|E_j|}
\int_{E_j}
k_j^{p(x)/k_j}\,dx.
\]
On $E_j$ one has $p(x)\geq k_j$, and therefore $k_j^{p(x)/k_j}\geq k_j$.  Hence
\[
\rho_p(v)
\geq
\sum_{j=1}^{\infty}k_j
=
\infty.
\]
On the other hand, let $0<\lambda<1$. Since
$k_j\leq p(x)<k_j+1$ on $E_j$,
\[
\begin{aligned}
\rho_p(\lambda v)
&=
\sum_{j=1}^{\infty}
\frac{1}{|E_j|}
\int_{E_j}
\lambda^{p(x)}
k_j^{p(x)/k_j}\,dx\\
&\leq
\sum_{j=1}^{\infty}
\lambda^{k_j}
k_j^{(k_j+1)/k_j}.
\end{aligned}
\]
Since  $k_j^{(k_j+1)/k_j}\leq k_j^2$  and the $k_j$ are distinct positive integers,
\[
\sum_{j=1}^{\infty}
\lambda^{k_j}k_j^2
\leq
\sum_{m=1}^{\infty}\lambda^m m^2
<\infty.
\]
Thus $\rho_p(\lambda v)<\infty$, for every $0<\lambda<1$, and consequently $v\in L^{p(\cdot)}(\Omega)$.\\
Furthermore,
\[
\rho_p(\lambda v)\longrightarrow0
\qquad\text{as }\lambda\downarrow0.
\]
Indeed, for $0<\lambda\leq 1/2$ the preceding series is dominated by
the convergent series
\[
\sum_{j=1}^{\infty}2^{-k_j}k_j^2,
\]
and the conclusion follows from dominated convergence.\\
Fix now $\delta>0$. Choose $\varepsilon\in(0,1)$ sufficiently close
to $1$ so that
\[
\rho_p((1-\varepsilon)v)<\delta.
\]
Then  $\varepsilon v\in B_{\rho_p,\delta}(v)$.  For $N\geq1$, define
\[
v_N
=
\sum_{j=1}^{N}
k_j^{1/k_j}
|E_j|^{-1/p(x)}
{\mathbbm 1}_{E_j}.
\]
For any fixed $\varepsilon\in(0,1)$,
\[
\begin{aligned}
\rho_p\bigl(\varepsilon(v-v_N)\bigr)
&=
\sum_{j=N+1}^{\infty}
\frac{1}{|E_j|}
\int_{E_j}
\varepsilon^{p(x)}
k_j^{p(x)/k_j}\,dx\\
&\leq
\sum_{j=N+1}^{\infty}
\varepsilon^{k_j}k_j^2
\longrightarrow0
\end{aligned}
\]
as $N\to\infty$. Hence  
$$\varepsilon v_N  \mathop{\longrightarrow}^{\rho_p}
\varepsilon v.$$  
However, for every $N$,  we have  $\rho_p(v-\varepsilon v_N)=\infty$.  Indeed, on every $E_j$ with $j>N$ one has
$v-\varepsilon v_N=v$, and therefore
\[
\rho_p(v-\varepsilon v_N)
\geq
\sum_{j=N+1}^{\infty}
\frac{1}{|E_j|}
\int_{E_j}
k_j^{p(x)/k_j}\,dx
\geq
\sum_{j=N+1}^{\infty}k_j
=
\infty.
\]
Consequently, $\varepsilon v_N\notin B_{\rho_p,\delta}(v)$, for every $N \geq 1$. Thus $\varepsilon v$ belongs to $B_{\rho_p,\delta}(v)$ but is the
$\rho_p$-limit of a sequence lying entirely outside that ball.
By Theorem~\ref{convergencecharacterization}, $B_{\rho_p,\delta}(v)$ cannot be $\tau_{\rho_p}$-open.  Finally, for any $w\in L^{p(\cdot)}(\Omega)$,
\[
B_{\rho_p,\delta}(w)
=
w-v+B_{\rho_p,\delta}(v).
\]
Since translations are homeomorphisms for the modular topology by
Proposition~\ref{basic-properties}(1), no modular ball of radius
$\delta$ is $\tau_{\rho_p}$-open. Since $\delta>0$ was arbitrary,
no modular ball in $L^{p(\cdot)}(\Omega)$ is
$\tau_{\rho_p}$-open.
\end{proof}

\smallskip
The following lemma is well known.
\begin{lemma}\label{closedballsclosed}
For any $\mathbf z\in\ell^{(p_n)}$ and $\delta>0$, the set
\[
B^{\ast}_{\rho_{\mathbf p},\delta}(\mathbf z)
=\{\mathbf y\in\ell^{(p_n)}:\rho_{\mathbf p}(\mathbf y-\mathbf z)\le\delta\}
\]
is closed in the modular topology. Its continuous analogue is also true
in $L^{p(\cdot)}(\Omega)$.
\end{lemma}
\begin{proof}
Suppose $\mathbf x^j$ lies in the displayed set and
$\mathbf x^j\xrightarrow{\rho_{\mathbf p}}\mathbf x$. Coordinatewise
convergence follows, and Fatou's lemma gives
\[
\rho_{\mathbf p}(\mathbf x-\mathbf z)
\le\liminf_j\rho_{\mathbf p}(\mathbf x^j-\mathbf z)\le\delta.
\]
For the continuous analogue, modular convergence implies convergence in
measure on every finite-measure subset: split that subset into
$\{p\le M\}$ and $\{p>M\}$, use Chebyshev's inequality on the first,
and then let $M\to\infty$. A diagonal argument supplies an
almost-everywhere convergent subsequence, and the same Fatou argument
applies in $L^{p(\cdot)}(\Omega)$.
\end{proof}

\begin{proposition}\label{closure-of-modular-ball}
For either variable-exponent modular and every $z$ and $\delta>0$,
\[
\overline{B_{\rho_p,\delta}(z)}^{\,\rho_p}
=\{y:\rho_p(y-z)\le\delta\}.
\]
\end{proposition}
\begin{proof}
Lemma~\ref{closedballsclosed} gives the inclusion from left to right.
If $\rho_p(y-z)\le\delta$ and $0<t<1$, put
$y_t=z+t(y-z)$.  Convexity gives
$\rho_p(y_t-z)\le t\delta<\delta$, while
\[
\rho_p(y_t-y)=\rho_p((1-t)(y-z))
\le(1-t)\rho_p(y-z)\longrightarrow0
\]
as $t\uparrow1$.  Hence $y$ lies in the modular closure of the open
ball, proving the reverse inclusion.
\end{proof}

\smallskip

It follows immediately from Proposition \ref{basic-properties}(1) that:

\begin{corollary}\label{multiplesofclosedballs}
Let $\overline{B_{\rho_{\mathbf{p}},\delta}(\mathbf{z})}$ denote any closed ball centered at ${\mathbf z}$ in either $\ell^{(p_n)}$ or $L^{p(\cdot)}(\Omega)$. Then for any $j\geq 1$, the set $j\overline{B_{\rho_{\mathbf{p}},\delta}(\mathbf{z})}$ is modularly closed. Indeed, for any modularly closed set $C$, the complement of $jC$ is $j(X_\rho\setminus C)$, which
is open by Proposition~\ref{basic-properties}(2).
\end{corollary}
\begin{definition}
Given a modular space $(X,\rho)$ and a subset $Y\subseteq X_{\rho}$, the modular diameter of $Y$ is defined as $$\text{diam}_{\rho}(Y):=\sup\{\rho(a-b),a,b\in Y \}.$$
\end{definition}

\smallskip

\begin{corollary}\label{infinitediameter}
If $(p_n)$ is unbounded, every nonempty modularly open subset of
$\ell^{(p_n)}$ has infinite modular diameter. The same holds in
$L^{p(\cdot)}(\Omega)$ when $\Omega$ is bounded, $p$ is finite almost
everywhere, and $p^+=\infty$. 
\end{corollary}
\begin{proof}
Let ${\mathbf s}\in \ell^{(p_n)}$ be the sequence introduced in Theorem \ref{modularballsnotopen}. Consider first an open set \(\mathcal O\) such that ${\mathbf s}\in \mathcal O$; since $\mathcal O$ is open, for some $\delta>0$ the modular ball \(B_{\rho_{\mathbf{p}},\delta}(\mathbf{s})\) is contained in $\mathcal O$. On account of Theorem \ref{modularballsnotopen}, one can fix $\epsilon>0$ such that $(1-\epsilon){\mathbf s}\in B_{\rho_{\mathbf{p}},\delta}(\mathbf{s})$. Since \({\mathcal O}\) is open by assumption, there must exist a modular ball $B_{\rho_{\mathbf{p}},\eta}((1-\epsilon)\mathbf{s})$ that is contained in \({\mathcal O}\).  It was shown in Theorem \ref{modularballsnotopen} that any modular ball centered at $(1-\epsilon){\mathbf s}$ must contain an element \(\mathbf{t}\) such that \(\rho_{\mathbf{p}}(\mathbf{s} - \mathbf{t}) = \infty\).  In conclusion, since ${\mathbf t} \in \mathcal O$ and ${\mathbf s} \in \mathcal O$, one has
\begin{equation*}
\text{diam}_{\rho_{{\mathbf p}}}({\mathcal O})\geq \rho_{{\mathbf p}}\left({\mathbf t}-{\mathbf s}\right)=\infty.
\end{equation*}
To move on to the general case,  if \(\mathcal O\) is an arbitrary modularly open subset of $\ell^{(p_n)}$ and \(\mathbf{w} \in \mathcal{O}\), the set \(\mathcal{O} + \{\mathbf{s} - \mathbf{w}\}\) is modularly open and contains \(\mathbf{s}\), and hence, a modular ball \(B_{\rho_{\mathbf{p}},\delta}(\mathbf{s})\). It follows that
$$\text{diam}_{\rho_{{\mathbf p}}} \left(\mathcal{O} + \{\mathbf{s} - \mathbf{w}\}\right) = \infty,$$
and since the \(\rho\)-diameter is translation-invariant, \(\mathcal{O}\) must have infinite diameter.\\
For the function space, use the function $v$ constructed in
Theorem~\ref{general-unbounded-example}. If an open set contains
$v$, it also contains $\varepsilon v$ for some $\varepsilon\in(0,1)$
close to $1$. A modular ball about $\varepsilon v$ contained in the open
set contains $\varepsilon v_N$ for all sufficiently large truncations
$v_N$, while $\rho_p(v-\varepsilon v_N)=\infty$. Translation again
handles an arbitrary nonempty open set.
\end{proof}

We first establish a separability property of the modular topology that will
play an important role in the sequel.

\begin{theorem}\label{separability}
If $(p_n)$ is unbounded, then the subspace ${\mathbf c}_{00}$ consisting of all sequences which vanish at all but a finite number of terms is dense in
$(\ell^{(p_n)},\tau_{\rho_{\mathbf{p}}})$. Hence, $(\ell^{(p_n)},\tau_{\rho_{\mathbf{p}}})$ is separable.
Analogously, if $\Omega\subset\mathbb{R}^n$ is a domain and
$p\in C(\Omega)$, possibly with $p^+=\infty$, then $C_0^\infty(\Omega)$ is $\tau_{\rho_p}$-dense in
$L^{p(\cdot)}(\Omega)$. It follows that
$(L^{p(\cdot)}(\Omega),\tau_{\rho_p})$ is separable
if
\[
1\leq p(x)<\infty
\]
for every $x\in\Omega$, even when $p$ is unbounded, i.e.,
$p^+=\infty$.
\end{theorem}
\begin{proof}
Denote the modular closure of ${\mathbf c}_{00}$ by
$\overline{{\mathbf c}_{00}}^{\rho_{\mathbf p}}$. Then
$\overline{{\mathbf c}_{00}}^{\rho_{\mathbf p}}$ is a
$\rho_{\mathbf p}$-closed subspace of $\ell^{(p_n)}$.\\
Suppose first that
$\mathbf a\in\ell^{(p_n)}$ satisfies
$\rho_{\mathbf p}(\mathbf a)<\infty$. For every $\delta>0$,
there exists a finitely supported sequence
$\mathbf s\in {\mathbf c}_{00}$ such that
\[
\rho_{\mathbf p}(\mathbf a-\mathbf s)<\delta.
\]
Thus every modular ball centered at $\mathbf a$ intersects ${\mathbf c}_{00}$.
Since every $\tau_{\rho_{\mathbf p}}$-neighborhood of
$\mathbf a$ contains a modular ball centered at $\mathbf a$,
every $\tau_{\rho_{\mathbf p}}$-neighborhood of $\mathbf a$
intersects ${\mathbf c}_{00}$. Hence
\[
\mathbf a\in\overline{{\mathbf c}_{00}}^{\rho_{\mathbf p}}.
\]
Now let $\mathbf b\in\ell^{(p_n)}$ be arbitrary. By the definition
of $\ell^{(p_n)}$, there exists $\lambda>0$ such that
\[
\rho_{\mathbf p}(\lambda\mathbf b)<\infty.
\]
The preceding argument gives
$\lambda\mathbf b\in\overline{{\mathbf c}_{00}}^{\rho_{\mathbf p}}$.
Since $\overline{{\mathbf c}_{00}}^{\rho_{\mathbf p}}$ is a vector subspace and
$\lambda>0$, it follows that
\[
\mathbf b\in\overline{{\mathbf c}_{00}}^{\rho_{\mathbf p}}.
\]
Therefore ${\mathbf c}_{00}$ is $\tau_{\rho_{\mathbf p}}$-dense in
$\ell^{(p_n)}$. Since the finitely supported sequences with rational
coordinates form a countable  modularly dense subset of ${\mathbf c}_{00}$, it follows that
$(\ell^{(p_n)},\tau_{\rho_{\mathbf p}})$ is separable.\\
We now turn to $L^{p(\cdot)}(\Omega)$. Let
$f\in L^{p(\cdot)}(\Omega)$ satisfy
\[
\rho_p(f)
=
\int_\Omega |f(x)|^{p(x)}\,dx<\infty,
\]
and fix $\delta>0$. For $M>0$, set
\[
\Omega_M:=\{x\in\Omega:p(x)<M\}.
\]
Since $p\in C(\Omega)$, the set $\Omega_M$ is open. Moreover,
$\Omega_M\uparrow\Omega$ as $M\to\infty$. Since
$|f(\cdot)|^{p(\cdot)}\in L^1(\Omega)$, the dominated convergence
theorem yields
\[
\int_{\Omega\setminus\Omega_M}
|f(x)|^{p(x)}\,dx\longrightarrow0
\qquad\text{as }M\to\infty.
\]
Hence $M$ may be chosen sufficiently large so that
\[
\int_{\Omega\setminus\Omega_M}
|f(x)|^{p(x)}\,dx<\frac{\delta}{2}.
\]
On $\Omega_M$ the exponent $p$ is bounded. Therefore the classical
density result \cite[Theorem~2.11]{KR} applies, and there exists
$\phi\in C_0^\infty(\Omega_M)$ such that
\[
\int_{\Omega_M}
|f(x)-\phi(x)|^{p(x)}\,dx<\frac{\delta}{2}.
\]
Extending $\phi$ by zero outside $\Omega_M$, we obtain
$\phi\in C_0^\infty(\Omega)$. Since $\phi=0$ on
$\Omega\setminus\Omega_M$, we have
\[
\begin{aligned}
\rho_p(f-\phi)
&=
\int_{\Omega_M}
|f(x)-\phi(x)|^{p(x)}\,dx
+
\int_{\Omega\setminus\Omega_M}
|f(x)-\phi(x)|^{p(x)}\,dx\\
&=
\int_{\Omega_M}
|f(x)-\phi(x)|^{p(x)}\,dx
+
\int_{\Omega\setminus\Omega_M}
|f(x)|^{p(x)}\,dx\\
&<\frac{\delta}{2}+\frac{\delta}{2}
=\delta.
\end{aligned}
\]
Thus
\[
B_{\rho_p,\delta}(f)\cap C_0^\infty(\Omega)\neq\emptyset.
\]
Since $\delta>0$ was arbitrary, it follows that
$f$ belongs to the modular closure of $C_0^\infty(\Omega)$.

Finally, let $g\in L^{p(\cdot)}(\Omega)$ be arbitrary. By the
definition of $L^{p(\cdot)}(\Omega)$, there exists $\lambda>0$
such that
\[
\rho_p(\lambda g)
=
\int_\Omega |\lambda g(x)|^{p(x)}\,dx<\infty.
\]
By the preceding argument, $\lambda g$ belongs to the modular closure
of $C_0^\infty(\Omega)$. Since this closure is a vector subspace,
it follows that $g$ also belongs to it. Consequently,
\[
\overline{C_0^\infty(\Omega)}^{\,\rho_p}
=
L^{p(\cdot)}(\Omega).
\]
Since $p$ is bounded on every compact subset of $\Omega$, the separability of $L^{p(\cdot)}(\Omega)$ in the modular topology
then follows from \cite[Corollary~2.12]{KR} by a straightforward exhaustion argument.
\end{proof}

\smallskip

\begin{remark}
{\normalfont
To the authors' best knowledge, Theorem~\ref{separability} is new. It is known that the norm
topologies of both $\ell^{(p_n)}$ and $L^{p(\cdot)}(\Omega)$ are
separable when $p^+<\infty$ and nonseparable when $p^+=\infty$;
see~\cite{KR}. In contrast, Theorem~\ref{separability} shows that
the modular topology remains separable even when $p^+=\infty$.
}
\end{remark}

\begin{theorem}\label{dualimplication}
For either variable-exponent space, every $\tau_{\rho_p}$-continuous
linear functional is continuous for the Luxemburg norm.  Moreover, if
$(p_n)$ is unbounded, then
\[
(\ell^{(p_n)},\tau_{\rho_{\mathbf p}})^*
\subsetneq(\ell^{(p_n)},\|\cdot\|_{\rho_{\mathbf p}})^*.
\]
If $p\in C(\Omega)$ is finite and unbounded, the analogous strict
inclusion holds for $L^{p(\cdot)}(\Omega)$.
\end{theorem}
\begin{proof}
Since every $\tau_\rho$-open set is norm open it follows that any
$\tau_\rho$-continuous linear functional is norm continuous.

For the sequence space, let $\mathbf s$ be the sequence used in
Theorem~\ref{modularballsnotopen}.  For every
$\mathbf a\in\mathbf c_{00}$ one has
$\rho_{\mathbf p}(\mathbf s-\mathbf a)=\infty$, so
$\|\mathbf s-\mathbf a\|_{\rho_{\mathbf p}}>1$ by Proposition
\ref{standard}$(ii)$.  Thus the norm closure of $\mathbf c_{00}$ is a
proper closed subspace.  Hahn--Banach supplies a nonzero norm-continuous
functional that vanishes on $\mathbf c_{00}$.  It cannot be
$\tau_{\rho_{\mathbf p}}$-continuous because $\mathbf c_{00}$ is
$\tau_{\rho_{\mathbf p}}$-dense by Theorem~\ref{separability}.

For the function space, use $v$ from
Theorem~\ref{general-unbounded-example}.  If
$h\in C_0^\infty(\Omega)$, continuity of $p$ bounds the exponent on
$\operatorname{supp}h$; all sufficiently high level sets $E_j$ in the
construction of $v$ therefore miss that support.  Consequently
$\rho_p(v-h)=\infty$ and $\|v-h\|_{\rho_p}>1$.  The norm closure of
$C_0^\infty(\Omega)$ is proper, and the same Hahn--Banach and density
argument proves strictness.
\end{proof}

The preceding results also yield a stronger conclusion concerning
modular balls. In the unbounded-exponent case, modular balls not only
fail to be open; they have empty interior in the modular topology.

\begin{corollary}\label{emptyinterior}
If the sequence $\mathbf p=(p_n)$ is unbounded, i.e., $p^+=\infty$,
then every modular ball in $\ell^{(p_n)}$ has empty
$\tau_{\rho_{\mathbf p}}$-interior.\\
Analogously, let $\Omega\subset\mathbb{R}^n$ be a bounded domain and let
$p\in C(\Omega)$ satisfy
\[
1\leq p(x)<\infty
\qquad\text{for every }x\in\Omega,
\]
with $p^+=\infty$. Then every modular ball in
$L^{p(\cdot)}(\Omega)$ has empty interior in the modular topology
$\tau_{\rho_p}$.
\end{corollary}
\begin{proof}
It was shown in Theorem \ref{general-unbounded-example} that for any $\delta>0$, the modular ball $B_{\rho_{\mathbf{p}},\delta}(\mathbf{s})$ in $\ell^{(p_n)}$, has no interior point $\mathbf{x}$ with $\rho_{\mathbf{p}}(\mathbf{x})<\infty.$ Select $\mathbf{y}\in B_{\rho_{\mathbf{p}},\delta}({\mathbf {s}})$. On account of Theorem \ref{separability}, any $\rho_{\mathbf{p}}$-open set $A$ containing $\mathbf{y}$ must contain a sequence with only finitely many nonzero terms. Pick one such sequence, say $\mathbf{x}$; then  $\mathbf{x}$ is not an interior point of $B_{\rho_{\mathbf{p}},\delta}(\mathbf{s})$, i.e., some point of $A$ must be in the complement of $B_{\rho_{\mathbf{p}},\delta}(\mathbf{s})$. It follows that $B_{\rho_{\mathbf{p}},\delta}(\mathbf{s})$ has empty interior. On account of Proposition \ref{basic-properties} (1), any ball
$B_{\rho_{\mathbf{p}},\delta}({\mathbf x})=\mathbf{x}-\mathbf{s}+B_{\rho_{\mathbf{p}},\delta}(\mathbf{s})$ must have empty interior.\\
We move now to the continuous case. Let $v$ be the function constructed in Theorem \ref{general-unbounded-example} and fix
$\delta>0$. We first show that $B_{\rho_p,\delta}(v)$ has empty interior.  Let  $g\in B_{\rho_p,\delta}(v)$  and let $A$ be any $\tau_{\rho_p}$-open set containing $g$.
By Theorem \ref{separability}, $C_0^\infty(\Omega)$ is dense in
$(L^{p(\cdot)}(\Omega),\tau_{\rho_p})$. Hence there exists
\[
h\in A\cap C_0^\infty(\Omega).
\]
Since $h$ has compact support and $p\in C(\Omega)$, the exponent $p$
is bounded on $\operatorname{supp} h$. On the other hand, in the
construction of $v$ in Theorem \ref{general-unbounded-example}, the sets
\[
E_j=\{x\in\Omega:k_j\leq p(x)<k_j+1\}
\]
satisfy $k_j\to\infty$. Therefore, for all sufficiently large $j$,  $E_j\cap\operatorname{supp} h=\varnothing$. Consequently, $h=0$ on $E_j$ for all sufficiently large $j$, and hence
\[
\begin{aligned}
\rho_p(v-h)
&\geq
\sum_{j\geq j_0}
\frac{1}{|E_j|}
\int_{E_j}
k_j^{p(x)/k_j}\,dx\\
&\geq
\sum_{j\geq j_0} k_j
=
\infty.
\end{aligned}
\]
Thus
\[
h\notin B_{\rho_p,\delta}(v).
\]
Since every $\tau_{\rho_p}$-open neighborhood $A$ of every
$g\in B_{\rho_p,\delta}(v)$ contains a point outside
$B_{\rho_p,\delta}(v)$, the modular ball
$B_{\rho_p,\delta}(v)$ has empty interior.\\
Finally, for any $w\in L^{p(\cdot)}(\Omega)$,
\[
B_{\rho_p,\delta}(w)
=
w-v+B_{\rho_p,\delta}(v).
\]
By Proposition~\ref{basic-properties}(1), translations are homeomorphisms of the modular
topology. Hence every modular ball of radius $\delta$ has empty
interior. Since $\delta>0$ was arbitrary, every modular ball in
$L^{p(\cdot)}(\Omega)$ has empty interior.
\end{proof}

\smallskip

\begin{remark}\label{closedballsemptyinterior}
{\normalfont
Assume that $(p_j)$ is unbounded ($p^+=\infty$) and $p\in C(\Omega)$ is unbounded. The same argument used in the proof of the preceding corollary shows that every closed
modular ball in either $\ell^{(p_n)}$ or $L^{p(\cdot)}(\Omega)$,
\[
\overline{B}_{\rho_{\mathbf p},\delta}(\mathbf t)
:=
\left\{
\mathbf y\in\ell^{(p_n)}:
\rho_{\mathbf p}(\mathbf t-\mathbf y)\leq\delta
\right\}
\]
or 
\[
\overline{B}_{\rho_{\mathbf p},\delta}(\mathbf t)
:=
\left\{
\mathbf y\in L^{p(\cdot)}(\Omega):
\rho_{\mathbf p}(\mathbf t-\mathbf y)\leq\delta
\right\}
\]
has empty interior in the modular topology. We present the proof for $\ell^{(p_n)}$, the continuous case being similarly handled. For any $\delta>0$, let $\mathbf s$ be the sequence introduced
in Theorem~\ref{modularballsnotopen}. By the density result in
Theorem~\ref{separability}, every nonempty open set intersects ${\mathbf c}_{00}$. On the other hand, by
the construction of $\mathbf s$,
\[
\rho_{\mathbf p}(\mathbf s-\mathbf z)=\infty
\qquad\text{for every }\mathbf z\in {\mathbf c}_{00}.
\]
Consequently,
\[
{\mathbf c}_{00}\cap\overline{B}_{\rho_{\mathbf p},\delta}(\mathbf s)=\varnothing,
\]
and hence
$\overline{B}_{\rho_{\mathbf p},\delta}(\mathbf s)$ has empty interior.
Since every modular ball is a translate of a modular ball centered at
$\mathbf s$, the same conclusion holds for every
$\overline{B}_{\rho_{\mathbf p},\delta}(\mathbf t)$.
}
\end{remark}

\smallskip

\begin{corollary}\label{firstcategoryl}
If the sequence $(p_n)$ is unbounded, the modular topological space $\left (\ell^{(p_n)},\tau_{\rho_{{\mathbf p}}}\right)$ is of the first category. Likewise, if $\Omega\subseteq {\mathbb R}^n$ is an open set and $p\in C(\Omega)$ is unbounded and $1\leq p(x)<\infty$, then $(L^{p(\cdot)}(\Omega),\tau_{\rho_p})$ is of the first category.
\end{corollary}
\begin{proof}

For either modular space,
\begin{equation}\label{union}
X_\rho=\bigcup_{j=1}^{\infty}j\overline{B}_{\rho,1}(0),
\end{equation}
because every $x\in X_\rho$ satisfies $\rho(x/j)\le1$ for some integer
$j$. Each member of the union is modularly closed by Corollary
\ref{multiplesofclosedballs}.

In the sequence case, let $\mathbf s$ be as in Theorem
\ref{modularballsnotopen}. The dense core $\mathbf c_{00}$ misses
$j\overline B_{\rho_{\mathbf p},1}(\mathbf s)$, because
$\rho_{\mathbf p}(\mathbf a/j-\mathbf s)=\infty$ for every
$\mathbf a\in\mathbf c_{00}$. Hence this closed set, and therefore its
translate $j\overline B_{\rho_{\mathbf p},1}(0)$, has empty interior.

In the function case, use $v$ from
Theorem~\ref{general-unbounded-example} and the dense core
$C_0^\infty(\Omega)$. For every $h$ in that core,
$\rho_p(h/j-v)=\infty$: continuity of $p$ bounds it on
$\operatorname{supp}h$, while the tail of $v$ lies on arbitrarily high
exponent levels. Thus $j\overline B_{\rho_p,1}(v)$ and its translate
$j\overline B_{\rho_p,1}(0)$ have empty interior. Equation
\eqref{union} proves that each space is of first category.
\end{proof}\medskip
\begin{theorem}\label{lphausdorff}
The modular topology $\tau_{\rho_{\mathbf p}}$ on the variable exponent
space $\ell^{(p_n)}$ is Hausdorff, even when the sequence $(p_n)$ is
unbounded. 
\end{theorem}

\begin{proof}
Fix $M\in\mathbb{N}$. For $x=(x_n)\in\ell^{(p_n)}$ and
$\varepsilon>0$, set
\[
U^M_{x,\varepsilon}
:=
\left\{
y=(y_n)\in\ell^{(p_n)}:
|x_M-y_M|<\varepsilon
\right\}.
\]
We first show that $U^M_{x,\varepsilon}$ is
$\tau_{\rho_{\mathbf p}}$-open. By the characterization of closed sets
in the modular topology, it is enough to prove that
$\ell^{(p_n)}\setminus U^M_{x,\varepsilon}$ is
$\tau_{\rho_{\mathbf p}}$-closed.

Let $(\psi^n)\subset\ell^{(p_n)}\setminus U^M_{x,\varepsilon}$ and
assume that
\[
\psi^n\xrightarrow{\rho_{\mathbf p}}\psi.
\]
Then, for every $\delta>0$, there exists $N\geq1$ such that, for
$n>N$,
\[
\delta^{p_M}
>
\rho_{\mathbf p}(\psi^n-\psi)
=
\sum_{j=1}^{\infty}
|\psi_j^n-\psi_j|^{p_j}
\geq
|\psi_M^n-\psi_M|^{p_M}.
\]
Hence, for $n>N$,
\[
|\psi_M-x_M|
\geq
|\psi_M^n-x_M|-|\psi_M^n-\psi_M|
\geq
\varepsilon-\delta.
\]
Since $\delta>0$ is arbitrary, it follows that
\[
|\psi_M-x_M|\geq\varepsilon.
\]
Thus,
\[
\psi\in\ell^{(p_n)}\setminus U^M_{x,\varepsilon},
\]
and therefore $U^M_{x,\varepsilon}$ is
$\tau_{\rho_{\mathbf p}}$-open.

Now let $a,b\in\ell^{(p_n)}$ with $a\neq b$. Then
$a_M\neq b_M$ for some $M\in\mathbb{N}$. Setting
\[
\varepsilon=\frac{|a_M-b_M|}{2},
\]
the sets $U^M_{a,\varepsilon}$ and $U^M_{b,\varepsilon}$ are disjoint
$\tau_{\rho_{\mathbf p}}$-open neighborhoods of $a$ and $b$,
respectively. Hence $\tau_{\rho_{\mathbf p}}$ is Hausdorff.\\

\end{proof}
The following discussion prepares the ground for the analog of Theorem \ref{lphausdorff} in the continuous case.

\begin{lemma}\label{lqconvergence}
If $\Omega\subset {\mathbb R}^n$ is bounded, $q(x)\leq p(x)$ in $\Omega$ and $L^{p(\cdot)}(\Omega) \supset (u_j)\overset{\rho_p}{\rightarrow}u$, then there is a subsequence of $(u_j)$, say $(u_{j_k})$ that $\rho_q$ converges to $u$, that is 
\begin{equation*}
\int\limits_{\Omega}|u_{j_k}-u|^qdx\rightarrow 0\,\,\text{as}\,\,k\rightarrow \infty.
\end{equation*}
\end{lemma}
\begin{proof}
For any $a>0$, the Lebesgue measure $|\cdot|$ on the Borel subsets of $\Omega$ is absolutely continuous with respect to the measure defined by $A\rightarrow \nu_{a}(A)=\int\limits_{A}a^p\,dx$; accordingly, for arbitrary $\varepsilon>0$ there exists $\eta>0$ such that for every Borel set $A\subseteq \Omega$, it holds $\nu_{a}(A)<\eta \Rightarrow |A|<\epsilon$.  Now, for fixed $a$, due to the $\rho_p$-convergence of the sequence $(u_j)$ it is concluded that for any $\varepsilon>0$, there exists $I>0$ such that 
\begin{equation}
\nu_{a}(\{x:|(u-u_j)(x)|>a\})\leq  \rho_p(u-u_j)<\eta,
\end{equation}
whenever $j>I$.\\
Thus, $$\xi_j=|(u-u_j)(x)|\rightarrow 0$$ in measure and one can infer the existence of a subsequence of $(\xi_j)$, say $(\xi_{j_k})$ such that $$\xi_{j_k}(x)=|(u-u_{j_k})(x)|\rightarrow 0\,\,\,\text{a.e. in}\,\,\Omega.$$ Then
\begin{align}
\int\limits_{\Omega}|\xi_{j_k}|^qdx&=\left(\int\limits_{|\xi_{j_k}|\geq 1}+
\int\limits_{|\xi_{j_k}|<1}\right)|\xi_{j_k}|^qdx \\ \nonumber &\leq
\rho_p(\xi_{j_k})+ \int\limits_{|\xi_{j_k}|<1}\xi_{j_k}dx.
\end{align}
Letting $k\rightarrow \infty$, the first term tends to zero by assumption and the second one tends to zero by virtue of Lebesgue's dominated convergence theorem.
\end{proof}
\smallskip

We highlight the following lemma, which will be needed in the sequel:

\begin{lemma}\label{l1convergence}
If $L^{p(\cdot)}(\Omega) \supset (u_j)\overset{\rho_p}{\rightarrow}u$ then there is a subsequence of $(u_j)$, say $(u_{j_k})$ that converges to $u$ in $L^{1}(\Omega)$. 
\end{lemma}

The preceding lemma allows us to construct a $\rho$-open neighborhood of  $u_0\in L^p(\Omega)$. Specifically, we have:  
\begin{lemma}\label{open1}
For $u\in L^{p(\cdot)}(\Omega)$, the set 
\begin{equation}
U_{u,\epsilon}=\left\{v\in L^{p(\cdot)}(\Omega):\int\limits_{\Omega}|v-u|dx<\epsilon\right\}
\end{equation}
is $\tau_{\rho_p}$-open.
\begin{proof}
Take $L^{p(\cdot)}(\Omega) \setminus U_{u,\epsilon}\supset (u_j)\overset{\rho}\rightarrow v$. On account of Lemma \ref{l1convergence} there exists a subsequence $(u_{j_k})$ that converges to $v$ in $L^1$. Then
\begin{equation}
\int\limits_{\Omega}|v-u|dx\geq \int\limits_{\Omega}|u_{j_k}-u|dx-\int\limits_{\Omega}|u_{j_k}-v|dx,
\end{equation}
from which it is obvious that $v\in L^{p(\cdot)}(\Omega) \setminus U_{u,\epsilon}$.
\end{proof}
\end{lemma}
\begin{corollary}\label{lpqmodularcontinuity}
If $p(\cdot)$ and $q(\cdot)$ are measurable variable exponents on $\Omega$ with $1\leq q(x)\leq p(x)<\infty$, then if $\Omega\subset {\mathbb R}^n$ is a bounded domain, the inclusion
\begin{equation}
i_{p,q}:L^{p(\cdot)}(\Omega)\hookrightarrow L^{q(\cdot)}(\Omega)
\end{equation}
is modularly continuous.
\end{corollary}
\begin{proof}
Let $M\subseteq L^{q(\cdot)}(\Omega)$ be $\rho_q$-closed and consider a sequence $(v_j)\subset  i_{p,q}^{-1}(M)$ such that
\begin{equation}
(v_j)\overset{\rho_p}{\rightarrow}v.
\end{equation}
On account of Lemma \ref{lqconvergence} there exists a subsequence $(v_{j_k})$ of $(v_j)$ such that $v_{j_k}=i_{p,q}(v_{j_k})\overset{\rho_q}{\rightarrow v}=i_{p,q}v$.
 Since $M$ is $\rho_q$-closed, $v\in M$. Hence
$i_{p,q}^{-1}(M)$ is $\rho_p$-closed.
\end{proof}
\begin{proposition}\label{Lpboundedhausdorff}
If $\Omega\subset{\mathbb R}^n$ is a bounded domain and $1<p(x)<\infty$ then $\left(L^{p(\cdot)}(\Omega),\tau_{\rho_{p}}\right)$ is Hausdorff.
\end{proposition}
\begin{proof}
For two distinct elements $u_1$ and $u_2$ in $L^p(\Omega)$ it is clear that $U_{u_1,\frac{\|u_1-u_2\|_1}{2}}\cap U_{u_2,\frac{\|u_1-u_2\|_1}{2}}$ are disjoint open neighborhoods of $u_1$ and $u_2$ respectively.
\end{proof}
\begin{theorem}\label{Lphausdorff}
If $\Omega \subset {\mathbb R}^n$ is an arbitrary domain, then $\left(L^{p(\cdot)}(\Omega),\tau_{\rho_{p}}\right)$ is Hausdorff if $1<p(x)<\infty$. This includes the case $p^+=\infty$.
\end{theorem}

\begin{proof}
By translation it suffices to separate $0$ from a nonzero
$u\in L^{p(\cdot)}(\Omega)$. Choose $R>0$ so that the bounded open set
$E=\Omega\cap B_R(0)$ satisfies $T_Eu=u|_E\ne0$. The restriction map
$T_E:L^{p(\cdot)}(\Omega)\to L^{p(\cdot)}(E)$ is modularly continuous,
because $\rho_{p,E}(T_Ef)\le\rho_{p,\Omega}(f)$. Proposition
\ref{Lpboundedhausdorff} gives disjoint open neighborhoods $O_0,O_u$ of
$0$ and $T_Eu$. Their inverse images are the required disjoint modularly
open neighborhoods of $0$ and $u$.
\end{proof}
The only noteworthy cases of Theorems \ref{lphausdorff} and \ref{Lphausdorff} arise when the sequence \((p_n)\) is unbounded and the function $p(\cdot)$ is unbounded, respectively. For a bounded exponent sequence \((p_n)\), (for a bounded function $p(\cdot)$ on $\Omega$) it has been proved in Theorems \ref{r-continuous-p^+} and \ref{r-continuous-p^+L} that the modular \(\rho_{(p_n)}\) $(\rho_{p}$) satisfies the \(\Delta_2\)-condition, and according to Theorem \ref{Mainequivalence}, \(\tau_{\rho}\) corresponds to the topology induced by the Luxemburg norm, which is \(T_i\) for \(i = 0, 1, 2, 3, 4\).

Recall that if \((X,\tau)\) is a first-countable topological space and \(A \subset X\), then for any \(x\) in the \(\tau\)-closure of \(A\), there exists a sequence \((a_j) \subset A\) that \(\tau\)-converges to \(x\).

\begin{theorem}\label{non-first-countable}
Let $(p_n)$ be an unbounded sequence. Then the modular topology on
$\ell^{(p_n)}$ is not first-countable.  Moreover, let $\Omega\subset\mathbb{R}^n$ be a domain and let
$p\in C(\Omega)$ satisfy
\[
1\leq p(x)<\infty
\qquad\text{for every }x\in\Omega,
\]
with $p^+=\infty$. Then the modular topology on
$L^{p(\cdot)}(\Omega)$ is not first-countable.
\end{theorem}

\begin{proof}
We begin with the sequence space. Since $(p_n)$ is unbounded, there exists
a strictly increasing sequence $(n_k)$ such that
\[
p_{n_k}\geq k
\qquad\text{for every }k\geq1.
\]
Let $A\subset\ell^{(p_n)}$ be the subspace consisting of sequences with
only finitely many nonzero terms. By Proposition~\ref{subspace},
$\overline{A}^{\rho_{\mathbf p}}$ is a vector subspace.\\
Set
\[
\mathbf{x}
=
\frac12\,{\mathbbm 1}_{\{n_1,n_2,\ldots\}} .
\]
If
\[
\mathbf{x}^{(N)}
=
\frac12\,{\mathbbm 1}_{\{n_1,\ldots,n_N\}},
\]
then $\mathbf{x}^{(N)}\in A$ and
\[
\rho_{\mathbf p}\bigl(\mathbf{x}^{(N)}-\mathbf{x}\bigr)
=
\sum_{k=N+1}^{\infty}2^{-p_{n_k}}
\leq
\sum_{k=N+1}^{\infty}2^{-k}
\longrightarrow0.
\]
Hence  $\mathbf{x}\in\overline{A}^{\rho_{\mathbf p}}$.  Since $\overline{A}^{\rho_{\mathbf p}}$ is a vector subspace,
\[
\mathbf{s}:=2\mathbf{x}
=
{\mathbbm 1}_{\{n_1,n_2,\ldots\}}
\in\overline{A}^{\rho_{\mathbf p}}.
\]
However, $\mathbf{s}$ cannot be the modular limit of any sequence in $A$.
Indeed, for every $\mathbf{a}\in A$, there are infinitely many indices
$n_k$ outside the support of $\mathbf{a}$, and therefore
\[
\rho_{\mathbf p}(\mathbf{s}-\mathbf{a})
\geq
\sum_{n_k\notin\operatorname{supp}\mathbf{a}}1
=
\infty.
\]
Thus $\mathbf{s}$ belongs to the closure of $A$ but is not the limit of
any sequence in $A$.\\
If $\tau_{\rho_{\mathbf p}}$ were first-countable, every point in the
closure of a set would be the limit of a sequence from that set. This
contradiction proves that the modular topology on $\ell^{(p_n)}$ is not
first-countable.\\
We turn to $L^{p(\cdot)}(\Omega)$. Let $v$ be the function constructed
in Proposition~6.1. By Theorem~6.4,
$C_0^\infty(\Omega)$ is dense in the modular topology, so
\[
v\in\overline{C_0^\infty(\Omega)}^{\rho_p}.
\]
We claim that $v$ is not the modular limit of any sequence in
$C_0^\infty(\Omega)$.  Indeed, if $\phi\in C_0^\infty(\Omega)$, then
$\operatorname{supp}\phi$ is compact. Since $p\in C(\Omega)$,
$p$ is bounded on $\operatorname{supp}\phi$. In the construction of
$v$, the sets
\[
E_j=\{x\in\Omega:k_j\leq p(x)<k_j+1\}
\]
satisfy $k_j\to\infty$. Hence, for all sufficiently large $j$,  $E_j\cap\operatorname{supp}\phi=\varnothing$.  Therefore,
\[
\begin{aligned}
\rho_p(v-\phi)
&\geq
\sum_{j\geq j_0}
\frac{1}{|E_j|}
\int_{E_j}k_j^{p(x)/k_j}\,dx\\
&\geq
\sum_{j\geq j_0}k_j
=
\infty.
\end{aligned}
\]
Thus no sequence in $C_0^\infty(\Omega)$ can $\rho_p$-converge to $v$.  Hence $v$ belongs to the modular closure of $C_0^\infty(\Omega)$ but is
not the modular limit of any sequence from that set. First countability
would again force these two notions of closure to agree. Consequently,
the modular topology on $L^{p(\cdot)}(\Omega)$ is not first-countable.
\end{proof}

Since second countability implies first countability, it is immediate that:

\begin{corollary}  Under the assumptions of Theorem \ref{non-first-countable}, 
neither $(\ell^{(p_n)},\tau_{\rho_{\mathbf{p}}})$ nor $(L^{p(\cdot)}(\Omega),\tau_{\rho_{{p}}})$ are second countable.
\end{corollary}

\medskip
The preceding results show that, when $p^+=\infty$, the modular topology
is not merely a convenient language for modular convergence. It exhibits
a distinctive functional-analytic structure of its own. Although it is
no longer a topological vector space, it retains a fundamental
compatibility with the underlying linear structure: the modular closure
of every linear subspace is again a linear subspace. At the same time,
the topology is Hausdorff and separable while failing to be
first-countable; its modular balls have empty interior, and its nonempty
open sets have infinite modular diameter. Thus, the failure of the
$\Delta_2$-condition does not simply produce a weaker substitute for norm
convergence, but rather a genuinely different topology that preserves
essential linear features while exhibiting phenomena impossible in the
classical norm-topological setting. The variational applications considered in the next section should be viewed in this context: they illustrate how the modular topology remains analytically useful precisely in a regime where it differs from the classical norm topology.

\section{Applications }\label{applications}

The results established in the previous sections show that, in the absence of the $\Delta_2$-condition, the modular topology possesses structural properties that are fundamentally different from those of norm topologies. The natural question is whether these topological phenomena have genuine analytical consequences or are merely of independent topological interest.

The purpose of this section is to show that the modular topology arises
naturally in variational problems involving non-standard growth, in
particular those associated with the $p(\cdot)$-Laplacian. When the
exponent $p(\cdot)$ is unbounded, the modular topology differs from the
classical norm topology and more directly reflects the structure of the
underlying energy functional. In this setting, modular convergence
provides a natural framework for studying minimization problems and the
corresponding boundary value problems for the $p(\cdot)$-Laplacian.

The existence results obtained below therefore illustrate that the
modular topology is not simply an abstract topological construction.
Rather, it provides a natural analytical setting for variational
problems with non-standard growth, where the classical norm topology
may fail to capture the relevant variational structure.\\

\begin{definition}\label{sobolev}
Let $\Omega\subset {\mathbb R}^n$ be a bounded, $C^1$ domain and $p:\Omega\rightarrow (1,\infty)$. Assume $p^->1.$ Write
\begin{equation}
W^{1,p^-,p(\cdot)}(\Omega):=\left\{u\in L^{p^-}(\Omega):\nabla u \in (L^{p(\cdot)}(\Omega))^{n}\right\}
\end{equation}
and
\begin{align}
\rho_{p^-,p}&:W^{1,p^-,p(\cdot)}(\Omega)\rightarrow [0,\infty]
\\ \nonumber &\rho_{p^-,p}(u):=\int\limits_{\Omega}|u|^{p^-}dx+\int\limits_{\Omega}\frac{|\nabla u|^p}{p}dx,
\end{align}
where $|\nabla u|$ denotes the Euclidean norm of the gradient.\\
Then, $\rho_{p^-,p}$ is a convex, left-continuous modular on the vector space $W^{1,p^-,p(\cdot)}(\Omega)$.
\end{definition}
For vector fields $z$ put
\[
R_p(z):=\int_\Omega\frac{|z|^{p}}{p}\,dx.
\]
The following uniform-convexity property is the form needed below.
\begin{theorem}\label{ucmodular}
For every $\varepsilon>0$ there is $\delta(\varepsilon)>0$ such that
for all $z_1,z_2\in(L^{p(\cdot)}(\Omega))^n$, the inequality
\[
R_p\left(\frac{z_1-z_2}{2}\right)
\ge\varepsilon\frac{R_p(z_1)+R_p(z_2)}2
\]
implies
\[
R_p\left(\frac{z_1+z_2}{2}\right)
\le(1-\delta(\varepsilon))
\frac{R_p(z_1)+R_p(z_2)}2.
\]

\end{theorem}
\begin{proof}
This is the uniform convexity of the variable-exponent integral modular; see \cite[Theorem 5.4]{AOJA}.
\end{proof}
\begin{lemma}\label{weighted-modular-completeness}

Let $\Omega$ be bounded, let $p(x)<\infty$ almost everywhere, and let
$p_->1$. Then:
\begin{enumerate}
\item[(i)] $(L^{p(\cdot)}(\Omega))^n$ is complete for $R_p$-convergence;
\item[(ii)] $R_p(z_j)\to0$ implies
$\|z_j\|_{L^{p^-}(\Omega)}\to0$;
\item[(iii)] $\rho_{p^-,p}(u_j-u)\to0$ implies
$u_j\to u$ in $W^{1,p^-}(\Omega)$.
\end{enumerate}

\end{lemma}
\begin{proof}
First note that $R_p$ and $\rho_p$ generate the same modular space.  In
fact $R_p(z)\le\rho_p(z)$, while
\[
\rho_p(z/2)=\int_\Omega p(x)2^{-p(x)}
\frac{|z|^{p(x)}}{p(x)}\,dx\le C R_p(z),
\qquad C:=\sup_{t\ge1}t2^{-t}<\infty.
\]

We prove $(i)$ in the precise form needed here: if
$R_p(z_j-z_k)\to0$ as $j,k\to\infty$, then some
$z\in(L^{p(\cdot)}(\Omega))^n$ satisfies $R_p(z_j-z)\to0$.
Convergence to zero in $R_p$ implies convergence in measure.  Indeed,
for $\alpha>0$ and $M>1$, on
$\{|h|>\alpha\}\cap\{p\le M\}$ one has
\[
\frac{|h|^{p(x)}}{p(x)}\ge
c_{\alpha,M}:=\frac{\min\{\alpha,\alpha^M\}}{M}>0,
\]
whereas $|\{p>M\}|\to0$ as $M\to\infty$ because $p$ is finite almost
everywhere and $|\Omega|<\infty$.  Thus an $R_p$-Cauchy sequence is
Cauchy in measure.  Choose a subsequence $z_{j_k}\to z$ almost
everywhere.  Fatou's lemma gives, for each $j$,
\[
R_p(z_j-z)\le\liminf_{k\to\infty}R_p(z_j-z_{j_k}).
\]
The Cauchy property makes the left-hand side tend to zero with $j$.
The preceding comparison with $\rho_p$ also shows that $z$ belongs to
the same modular space.

For $(ii)$, $R_p(z_j)\to0$ first gives $z_j\to0$ in measure.  Put
$q=p^->1$ and fix $0<\alpha<1$. Split $\Omega$ into the sets where
$|z_j|\le\alpha$, where $\alpha<|z_j|\le1+\alpha$, and where
$|z_j|>1+\alpha$.  Then
\[
\begin{aligned}
\int_\Omega|z_j|^qdx
&\le \alpha^q|\Omega|
+(1+\alpha)^q|\{|z_j|>\alpha\}|\\
&\quad+C_{\alpha,q}R_p(z_j),
\end{aligned}
\]
where
$C_{\alpha,q}:=\sup_{t\ge q}t(1+\alpha)^{q-t}<\infty$.
Letting $j\to\infty$ and then $\alpha\downarrow0$ proves
$\|z_j\|_{L^q}\to0$.

Finally, $(iii)$ follows by applying $(ii)$ to the gradient term and
using the constant-exponent modular--norm equivalence for the
zeroth-order term.
\end{proof}
Let $S\subseteq W^{1,p^-,p(\cdot)}(\Omega)$ be a subspace. Henceforth, $\overline{S}^{p^-,p(\cdot)}$ will stand for the $\rho_{p^-,p}$-closure of $S$. For example, one could take $S=C^{\infty}_0(\Omega)$ or $S$ to be the subspace $W^{1,p^-,p(\cdot)}_{comp}(\Omega)$ consisting of all compactly supported functions in $W^{1,p^-,p(\cdot)}(\Omega)$. It is well known that for variable $p$ one may have the strict inclusion
\begin{equation}
\overline{C^{\infty}_0(\Omega)}^{p^-,p(\cdot)}\subsetneq \overline{W^{1,p^-,p(\cdot)}_{comp}(\Omega)}^{p^-,p(\cdot)};
\end{equation}
we refer the interested reader to \cite{H} for the details.

\begin{theorem}
If $S$ is as in the previous paragraph, then there holds the $\rho_{p^-,p}$-continuous inclusion
\begin{equation}
\overline{S}^{p^-,p(\cdot)}\subseteq \overline{S}^{W^{1,p^-}(\Omega)}.
\end{equation}
Here, $\overline{S}^{W^{1,p^-}(\Omega)}$ denotes the closure of $S$
in the standard norm topology of $W^{1,p^-}(\Omega)$.
\end{theorem}
\begin{proof}
Part $(iii)$ of the preceding lemma shows that the identity from the
modular space into $W^{1,p^-}(\Omega)$ is sequentially
continuous. Theorem~\ref{continuitysequential} makes it continuous.
The inverse image of the norm-closed set
$\overline S^{W^{1,p^-}(\Omega)}$ is therefore modularly closed and
contains $S$; it must contain $\overline S^{p^-,p(\cdot)}$.
\end{proof}

\begin{corollary}\label{subsetofw1p0}
If $S$ consists of functions compactly supported in $\Omega$, then  there holds the $\rho_{p^-,p}$-continuous inclusion
\begin{equation}
\overline{S}^{p^-,p(\cdot)}\subseteq W^{1,p^-}_0(\Omega).
\end{equation}
\end{corollary}
\begin{proof}
The result follows by observing that any compactly supported function in $W^{1,p^-}_0(\Omega)$ can be approximated in the norm topology by $C^{\infty}_0(\Omega)$ functions.  
\end{proof}
\subsection{The Dirichlet integral}
Let $\Omega\subset {\mathbb R}^n$ be a bounded, $C^1$ domain, $p$ as in the previous section. In what follows it will be assumed that $p\in C(\Omega)$ and that $p^->1$. In particular, the results in this section cover the case $p^+=\sup\limits_{x\in \Omega}p(x)=+\infty.$
\begin{definition}
Let $\varphi \in W^{1,p^-,p(\cdot)}(\Omega)$ with $\int\limits_{\Omega}\frac{|\nabla \varphi|^p}{p}dx<\infty$ and $S\subset W^{1,p^-,p(\cdot)}(\Omega)$ be a subspace of compactly supported functions. Consider the functional
\begin{align}\label{Dirin}
F&: \overline{S}^{p^-,p(\cdot)}\rightarrow [0,\infty]\\ \nonumber
F(u)&=R_p(\nabla (u-\varphi))=\int\limits_{\Omega}\frac{|\nabla (u-\varphi)|^p}{p}dx.
\end{align}
\end{definition}

We move next to Theorem \ref{existenceofminimizer}. Its proof relies entirely on the topology generated by the modular $\rho_{p^-,p}$ on $W^{1,p^-,p(\cdot)}(\Omega)$. Theorem \ref{existenceofminimizer} generalizes \cite[Theorem 7.2]{AOJA} in two aspects: The functional (\ref{Dirin}) is defined on a wider class of functions and the existence of the minimizer does not require the condition $p^->n$.

\begin{theorem}\label{existenceofminimizer}
There exists a unique minimizer $u\in \overline{S}^{p^-,p(\cdot)}$ of the Dirichlet integral (\ref{Dirin}).
\end{theorem}

\begin{proof}
The value $d:=\inf_{\overline{S}^{p^-,p(\cdot)}}F$ is finite because $0\in \overline{S}^{p^-,p(\cdot)}$ and
$F(0)=R_p(\nabla\varphi)<\infty$. Choose $u_j\in \overline{S}^{p^-,p(\cdot)}$ with
$F(u_j)\to d$, and put
\[
a_j:=\nabla(u_j-\varphi).
\]
We first prove
\begin{equation}\label{half-gradient-cauchy}
R_p\left(\frac{a_j-a_k}{2}\right)\longrightarrow0
\qquad(j,k\to\infty).
\end{equation}
If $d=0$, this follows directly from convexity and evenness of $R_p$.
If $d>0$ and~\eqref{half-gradient-cauchy} fails, there are $\eta>0$
and arbitrarily large $j,k$ for which its left-hand side is at least
$\eta$. Since $R_p(a_j),R_p(a_k)\to d$, Theorem~\ref{ucmodular}, with
a fixed positive ratio $\varepsilon$, gives
\[
R_p\left(\frac{a_j+a_k}{2}\right)
\le(1-\delta)\frac{R_p(a_j)+R_p(a_k)}2<d
\]
for large $j,k$. But $(u_j+u_k)/2\in \overline{S}^{p^-,p(\cdot)}$, and the left-hand side is
$F((u_j+u_k)/2)$, contradicting the definition of $d$.

Set $g_j=a_j/2$. Lemma~\ref{weighted-modular-completeness}$(i)$ gives a
vector field $g$ such that
\begin{equation}\label{convergnabla}
R_p(g_j-g)\longrightarrow0.
\end{equation}
Moreover,
\[
g_j-g_k=\frac{\nabla u_j-\nabla u_k}{2}.
\]
Lemma~\ref{weighted-modular-completeness}$(ii)$,
Corollary~\ref{subsetofw1p0}, and Poincar\'e's inequality show that
$q_j:=u_j/2$ is Cauchy in $W^{1,p^-}_0(\Omega)$. Let
$q_j\to q$ in that space. Since
\[
g_j=\nabla q_j-\tfrac12\nabla\varphi,
\]
the $L^{p^-}$ limits identify
$g=\nabla q-\tfrac12\nabla\varphi$. Consequently
\[
\rho_{p^-,p}(q_j-q)
=\|q_j-q\|_{L^{p^-}}^{p^-}+R_p(g_j-g)\longrightarrow0.
\]
The subspace $\overline{S}^{p^-,p(\cdot)}$ is modularly closed, so $q\in \overline{S}^{p^-,p(\cdot)}$; since it is a vector
subspace, $u:=2q\in \overline{S}^{p^-,p(\cdot)}$. Its associated gradient is
\[
\nabla(u-\varphi)=2g.
\]

From~\eqref{convergnabla}, a subsequence of $g_j$ converges almost
everywhere to $g$. Fatou's lemma therefore yields
\[
F(u)=R_p(2g)
\le\liminf_jR_p(2g_j)
=\liminf_jF(u_j)=d.
\]
Thus $u$ is a minimizer.

If $u,v\in \overline{S}^{p^-,p(\cdot)}$ are two minimizers, strict convexity of
$z\mapsto |z|^{p(x)}$ shows that their midpoint has energy strictly
less than $d$ unless $\nabla u=\nabla v$ almost everywhere. In the
latter case $u-v\in \overline{S}^{p^-,p(\cdot)}\subset W^{1,p^-}_0(\Omega)$, and Poincar\'e's
inequality gives $u=v$. This proves uniqueness.
\end{proof}

\smallskip

Theorem~\ref{existenceofminimizer} yields the following result, which seems to benew in the literature. We emphasize that Theorem~7.2 in \cite{AOJA} does not cover the case $p_-\leq n$, nor does it address minimization of the Dirichlet integral over arbitrary subspaces $S$ consisting of compactly supported functions. The qualifier ``$S$-weak'' records the testing class and is
part of the statement; different choices of $S$ need not lead to the same
notion of solution.

\begin{corollary}\label{existencedp}
Let $\Omega\subset\mathbb{R}^n$ be a bounded $C^1$ domain, and let
$p\in C(\Omega)$ satisfy
\[
p^-:=\inf_{x\in\Omega}p(x)>1.
\]
Let $S$ be an arbitrary subspace consisting of compactly supported functions.
Then, for any $\varphi\in W^{1,p^-,p(\cdot)}(\Omega)$ satisfying
\[
\int_\Omega \frac{|\nabla\varphi|^{p}}{p}\,dx<\infty,
\]
there exists a unique $S$-weak solution in the affine class
$\varphi- \overline{S}^{p^-,p(\cdot)}$, denoted by
$w\in W^{1,p^-,p(\cdot)}(\Omega)$, of the Dirichlet problem
\begin{equation}\label{DP}
\begin{cases}
\Delta_{p(\cdot)}w
=
\operatorname{div}\left(
|\nabla w|^{p(\cdot)-2}\nabla w
\right)=0
& \text{in }\Omega,\\
w=\varphi
& \text{on }\partial\Omega.
\end{cases}
\end{equation}
More precisely, the solution of the corresponding minimization problem
yields a function $w$ satisfying
\begin{equation}\label{weakformulation}
\int_\Omega
|\nabla w|^{p-2}\nabla w\cdot\nabla\phi\,dx=0,
\qquad
\text{for every }\phi\in S.
\end{equation}
\end{corollary}
\begin{proof}
Let $u\in \overline{S}^{p^-,p(\cdot)}$ be the unique minimizer in
Theorem~\ref{existenceofminimizer} and set $w=\varphi-u$. Then
$w-\varphi=-u\in \overline{S}^{p^-,p(\cdot)}\subset W^{1,p^-}_0(\Omega)$, which is the asserted
boundary condition in the Sobolev-trace sense.

Fix $\phi\in S$. Its support is compact, so continuity of $p$ makes the
exponent bounded there. Both $\nabla w$ and $\nabla\phi$ have finite
unscaled $p(\cdot)$-modular on that compact set. The usual difference
quotient is therefore dominated by an integrable multiple of
$|\nabla w|^{p}+|\nabla\phi|^{p}$. Since $u+t\phi\in \overline{S}^{p^-,p(\cdot)}$ for every
$t\in\mathbb R$, differentiation at the minimum gives
\[
0=\left.\frac{d}{dt}F(u+t\phi)\right|_{t=0}
=-\int_\Omega |\nabla w|^{p-2}\nabla w\cdot\nabla\phi\,dx.
\]
Thus~\eqref{weakformulation} is well defined and holds for every
$\phi\in S$. Uniqueness here is uniqueness in the variational class
$\varphi-\overline{S}^{p^-,p(\cdot)}$; it is exactly the uniqueness established in
Theorem~\ref{existenceofminimizer}. No broader uniqueness claim for
formal weak solutions with a different testing class is intended.
\end{proof}
In the light of \cite[Theorem 7.2]{AOJA}, the following proposition is clear:
\begin{proposition}
Under the assumptions of Theorem \ref{existenceofminimizer}, if in addition $p^->n$, the minimizer belongs to the $\tau_{1,p}$-closure of the subspace $S$, where $\tau_{1,p}$ is the topology generated by the modular  $\rho_{1,p}$ on $W^{1,p(\cdot)}(\Omega)$, defined by
\begin{equation}
\rho_{1,p}(u)=\int\limits_{\Omega}\frac{|u|^p+|\nabla u|^p}{p}dx.
\end{equation}
\end{proposition}
\begin{remark}{\normalfont A fundamental observation is in order here.  
In spite of the standard look of the weak formulation (\ref{weakformulation}), a sensitive point should be underlined: the class of admissible test functions $\phi$ used in the weak formulation does matter. For example, a solution might satisfy (\ref{weakformulation}) for all $\phi \in C^{\infty}_0(\Omega)$ but not for all $\phi\in W^{1,p^-,p(\cdot)}_{comp}(\Omega)$. Hence, the problem (\ref{DP}) is ambiguous, a fact already observed before. The reason behind this phenomenon is the complicated density behavior of smooth functions in function spaces of variable exponent. See \cite{BDS} and the references therein for very pathological examples and counterexamples involving these ideas, as well as for a thorough discussion of Lavrentiev's phenomenon. }
\end{remark}

\begin{remark}{\normalfont

For any exponent $p$, let $V_0^{1,p(\cdot)}(\Omega)$ and
$U_0^{1,p(\cdot)}(\Omega)$ denote the $\rho_{1,p}$-closures of
$C_0^{\infty}(\Omega)$ and $W_{\mathrm{comp}}^{1,p(\cdot)}(\Omega)$,
respectively. An explicit example of a domain $\Omega$, a variable
exponent $p$ with $p^->n$, and a function
\[
\varphi\in U_0^{1,p(\cdot)}(\Omega)\setminus V_0^{1,p(\cdot)}(\Omega)
\]
was constructed in \cite{Has}.\\
For this example, the minimizer of the Dirichlet integral
\eqref{Dirin} over $U_0^{1,p(\cdot)}(\Omega)$, given by Theorem
\ref{existenceofminimizer}, is $\varphi$, whereas the minimizer over
$V_0^{1,p(\cdot)}(\Omega)$ obtained in \cite{AOJA} must necessarily be
different from $\varphi$. Thus, the two closure procedures lead to
different admissible classes and, consequently, to different minimizers
of the corresponding variational problem.\\
In particular, this phenomenon yields an example of a
$C_0^{\infty}(\Omega)$-weak solution $w$ of the Dirichlet problem
\eqref{DP} which is not a $W_{\mathrm{comp}}^{1,p(\cdot)}(\Omega)$-weak
solution of the same problem. Hence, the distinction between these
closure procedures has a direct effect on the corresponding notion of
weak solution. See also \cite{BDS,H} for further discussions and
examples on this theme.

}\end{remark}


\medskip\medskip

\begin{thebibliography}{999}

\bibitem{BDS}
A. Kh. Balci, L. Diening, M. Surnachev,
\emph{New examples on Lavrentiev gap using fractals},
Calc. Var. Partial Differential Equations \textbf{59} (2020), 180.
https://doi.org/10.1007/s00526-020-01818-1

%\bibitem{birnbaum_orlicz_1931}
%Z. Birnbaum, W. Orlicz,
%\emph{\"Uber die Verallgemeinerung des Begriffes der zueinander konjugierten Potenzen},
%Studia Math. \textbf{3} (1931), 1--67.

%\bibitem{chen}
%S. Chen,
%\emph{Geometry of Orlicz spaces},
%Dissertationes Math. \textbf{356} (1996), 4--205.

\bibitem{DHHR}
L. Diening, P. Harjulehto, P. H\"ast\"o, M. R\r{u}\v{z}i\v{c}ka,
\emph{Lebesgue and Sobolev Spaces with Variable Exponents},
Lecture Notes in Mathematics, Vol. 2017, Springer, Berlin, 2011.

%\bibitem{Diestel}
%J. Diestel,
%\emph{Sequences and Series in Banach Spaces},
%Graduate Texts in Mathematics, Vol. 92, Springer-Verlag, New York, 1984.

\bibitem{dudley}
R. M. Dudley,
\emph{On sequential convergence},
Trans. Amer. Math. Soc. \textbf{112} (1964), 483--507.
https://doi.org/10.1090/S0002-9947-1964-0175081-6


\bibitem{Hajji2013}
A. Hajji,
\emph{Modular spaces topology},
Applied Mathematics \textbf{4} (2013), 1296--1300.
https://doi.org/10.4236/am.2013.49175

\bibitem{Haryadi2025}
H. Haryadi, S. Solikhin,
\emph{Some properties of modular topology in the Orlicz sequence space},
J. Fundamental Mathematics and Applications \textbf{8} (2025), 128--137.
https://doi.org/10.14710/jfma.v0i0.26542


\bibitem{H}
P. Harjulehto,
\emph{Variable exponent Sobolev spaces with zero boundary values},
Math. Bohem. \textbf{132}(2) (2007), 125--136.
https://doi.org/10.21136/MB.2007.134186

%\bibitem{HHKV}
P. Harjulehto, P. H\"ast\"o, M. Koskenoja, S. Varonen,
%\emph{The Dirichlet energy integral and variable exponent Sobolev spaces with zero boundary values},
%Potential Anal. \textbf{25}(3) (2006), 205--222.
%https://doi.org/10.1007/s11118-006-9023-3

\bibitem{Has}
P. H\"ast\"o,
\emph{Counter-examples of regularity in variable exponent Sobolev spaces},
in: \emph{The $p$-Harmonic Equation and Recent Advances in Analysis},
Contemp. Math., Vol. 370, American Mathematical Society, Providence, RI,
2005, 133--143.
https://doi.org/10.1090/conm/370/06832

%\bibitem{kaminska}
%A. Kami\'nska,
%\emph{On uniform convexity of Orlicz spaces},
%Indag. Math. \textbf{44} (1982), 27--36.

%\bibitem{OAP}
%M. A. Khamsi, P. Kumam, O. M\'endez,
%\emph{From modular spaces to boundary value problems: A survey of recent advances},
%Carpathian J. Math. \textbf{41}(2) (2025), 425--440.
%https://doi.org/10.37193/CJM.2025.02.10

\bibitem{KK}
M. A. Khamsi, W. M. Kozlowski,
\emph{Fixed Point Theory in Modular Function Spaces},
Birkh\"auser, New York, 2015.

\bibitem{AOJA}
M. A. Khamsi, J. Lang, O. M\'endez, A. Nekvinda,
\emph{The non-homogeneous Dirichlet problem for the $p(x)$-Laplacian
with unbounded $p(x)$ on a smooth domain},
J. Differential Equations \textbf{434} (2025), 113316.
https://doi.org/10.1016/j.jde.2025.113316

\bibitem{KR}
O. Kov\'a\v{c}ik, J. R\'akosn\'ik,
\emph{On spaces $L^{p(x)}$ and $W^{k,p(x)}$},
Czechoslovak Math. J. \textbf{41}(4) (1991), 592--618.

\bibitem{Kozlowski2020}
W. M. Kozlowski,
\emph{On modulated topological vector spaces and applications},
Bull. Aust. Math. Soc. \textbf{101} (2020), 325--332.
https://doi.org/10.1017/S0004972719000716

\bibitem{kozlowski_book}
W. M. Kozlowski,
\emph{Modular Function Spaces},
Monographs and Textbooks in Pure and Applied Mathematics, Vol. 122,
Marcel Dekker, New York/Basel, 1988.

\bibitem{ML}
O. M\'endez, J. Lang,
\emph{Analysis on Function Spaces of Musielak--Orlicz Type},
Taylor \& Francis, 2018.

\bibitem{Luxemburg}
W. A. J. Luxemburg,
\emph{Banach Function Spaces},
Ph.D. Thesis, Technische Hogeschool te Delft, Delft, 1955.

\bibitem{M:1983}
J. Musielak,
\emph{Orlicz Spaces and Modular Spaces},
Lecture Notes in Mathematics, Vol. 1034, Springer-Verlag, Berlin, 1983.

\bibitem{nakano}
H. Nakano,
\emph{Modulared Semi-Ordered Linear Spaces},
Maruzen Co., Tokyo, 1950.

\bibitem{nakano3}
H. Nakano,
\emph{Topology of Linear Topological Spaces},
Maruzen Co. Ltd., Tokyo, 1951.

\bibitem{orlicz1931}
W. Orlicz,
\emph{\"Uber konjugierte Exponentenfolgen},
Studia Math. \textbf{3} (1931), 200--211.

%\bibitem{RaRu}
%K. R. Rajagopal, M. R\r{u}\v{z}i\v{c}ka,
%\emph{Modeling of electrorheological materials},
%Continuum Mech. Thermodyn. \textbf{13}(1) (2001), 59--78.
%https://doi.org/10.1007/s001610100034

%\bibitem{rao}
%M. M. Rao, Z. D. Ren,
%\emph{Theory of Orlicz Spaces},
%Monographs and Textbooks in Pure and Applied Mathematics, Vol. 146,
%Marcel Dekker, New York, 1991.

%\bibitem{riesz}
%F. Riesz,
%\emph{Untersuchungen \"uber Systeme integrierbarer Funktionen},
%Math. Ann. \textbf{69} (1910), 449--497.

%\bibitem{sharapudinov}
%I. I. Sharapudinov,
%\emph{Topology of the space ${\cal L}^{p(t)}([0,1])$},
%Mat. Zametki \textbf{26}(4) (1979), 613--632;
%English transl., Math. Notes \textbf{26}(4) (1979), 796--806.
%https://doi.org/10.1007/BF01159546

%\bibitem{waterman}
%D. Waterman, T. Ito, F. Barber, J. Ratti,
%\emph{Reflexivity and summability: The Nakano $\ell(p_i)$ spaces},
%Studia Math. \textbf{33}(2) (1969), 141--146.

\end{thebibliography}
\end{document}